\documentclass[12pt, reqno]{amsart}
\usepackage{amsmath,amssymb,amsfonts,amstext, amsthm, amscd}
\usepackage{verbatim}
\usepackage{graphicx}
\usepackage{color}
\usepackage{enumerate}
\usepackage{epstopdf}

\newtheorem{theorem}{Theorem}
\newtheorem{lemma}[theorem]{Lemma}
\newtheorem{proposition}[theorem]{Proposition}
\newtheorem{corollary}[theorem]{Corollary}
\newtheorem{conjecture}{Conjecture}
\newtheorem{question}[conjecture]{Question}

\newtheorem*{theorem-nonum}{Theorem 3}

\theoremstyle{definition}
\newtheorem*{definition*}{Definition}
\newtheorem{remark}[theorem]{Remark}
\newtheorem{example}[theorem]{Example}

\newcommand\ns[1]{ \left\{ {#1} \right\} }

\newcommand{\Z}{{\mathbb Z}}
\newcommand{\R}{{\mathbb R}}
\newcommand{\N}{{\mathbb N}}

\newcommand{\e}{\varepsilon}

\newcommand{\f}{\ell}

\newcommand{\F}{\mathcal{F}}
\newcommand{\E}{\mathbb{E}}
\newcommand{\borel}{\mathcal{B}}
\newcommand{\s}{\delta}
\newcommand{\fix}{{\Sigma}}

\newcommand{\muu}{\delta}

\newcommand\m{{\mu}}

\newcommand\B{{\mathcal B}}        
\newcommand{\les}{{\mathcal{L}}}

\newcommand\X{\Omega}

\newcommand\garbage[1]{}

\renewcommand{\P}{{\mathbb P}}

\newcommand{\dff}[1]{\textbf{\emph{#1}}}

\newcommand{\erk}{\hfill \ensuremath{\Diamond}} 

\DeclareMathOperator{\essinf}{ess \, inf}
\DeclareMathOperator{\esssup}{ess \, sup}

\DeclareMathOperator{\var}{var}

\DeclareMathOperator{\orb}{orb}

\newcommand{\num}[1]{\# | {#1}| }

\newcommand\ind{{\mathbf{1}}}
\newcommand\1{{\mathbf{1}}}

\newcommand{\bb}[1]{\mathbb{#1}}

\usepackage{bbm}

\newcommand{\nua}{\nu^{\alpha}}

\begin{document}
	\author{Zemer Kosloff and Terry Soo}
	\title[Bernoulli actions and factors]{The orbital equivalence of Bernoulli actions and their Sinai factors}

	\address{Einstein Institute of Mathematics, 
		Hebrew University of Jerusalem, 
		\indent Edmund J. Safra Campus, Givat Ram. Jerusalem, 9190401, Israel.}
	\email{zemer.kosloff@mail.huji.ac.il}
	\urladdr{http://math.huji.ac.il/~zemkos/}

	\address{Department of Statistical Science,
		University College London,
		 Gower \indent Street,
		London WC1E 6BT,
		United Kingdom.}
	\email{math@terrysoo.com}
	\urladdr{http://www.terrysoo.com}


\begin{abstract}
Given a countable amenable group $G$ and $\lambda \in (0,1)$, we give an elementary construction  of a  type-$\mathrm{III}_{\lambda}$  Bernoulli group action.  In the case where $G$ is the integers, we   show that our nonsingular Bernoulli shifts  have  independent and identically distributed  factors.  
\end{abstract}

\maketitle

\section{Introduction}

Let $(\Omega, \F, \mu)$ be a  $\sigma$-finite measure space.  We say that a measurable map $T: \Omega \to \Omega$ is \dff{nonsingular}  if the measure given by $\mu \circ T^{-1}$ is equivalent to $\mu$, 
that is, they have same null sets.  
The measure space endowed with a nonsingular map is a \dff{nonsingular dynamical system}, which is a model for studying dynamics of a system which is \emph{not}  at equilibrium; in the special more widely studied  case where $\mu(\Omega) =1$ and $\mu$ is invariant, $\mu \circ T^{-1} = \mu$, we obtain a \dff{measure-preserving} system which models dynamics at equilibrium.  If $T$ is invertible, then we also refer to it as an \dff{automorphism} and  also say the system is \dff{invertible}.       We say that $T$ is \dff{ergodic} if for all $E \in \F$, we have that if $\mu(E \triangle T^{-1}(E)) =0$, then $\mu(E)=0$ or $\mu(E^c) =0$.    The theory and stock of examples developed for the study of nonsingular dynamical systems has received considerable attention in the last decade, see the survey article by Danilenko and Silva \cite{dal-siv} and its updated version \cite{dal-siv-up}.  We hope to add to the stock of useful examples in the study of the isomorphism class of a nonsingular system.

Ergodic invertible nonsingular dynamical systems are usually classed by their Krieger ratio set \cite{kriegertypes,MR415341}, which are isomorphism invariants.    We will give more involved definitions in Section \ref{krs}.    We say that $T$ is of \dff{type-$\mathrm{II}$} if it admits an invariant $\sigma$-finite measure; if the invariant measure can be chosen to be finite, then $T$ is \dff{type-$\mathrm{II_1}$}, otherwise $T$ is \dff{type-$\mathrm{II_{\infty}}$}.  If $T$ is conservative and not type-$\mathrm{II}$, 
then we say it is 
\dff{type-$\mathrm{III}$}; type-$\mathrm{III}$ systems are further classified by a parameter $\lambda \in [0,1]$.     The existence of a type-$\mathrm{III}$ system was a long standing open problem of Halmos \cite{MR21963}, which was  resolved by Ornstein \cite{MR146350} who exhibited a nonsingular odometer system that  is of type-$\mathrm{III}$.  Ulrich Krengel and Benjamin Weiss   asked what
are the possible Krieger types of shift systems arising from independent sequences. 

Let $\N = \ns{0,1,2, \ldots}$ be the set of natural numbers.
Let $A$ be a set which will usually be finite, countable,  or a subset of $\R$ and let $(\rho_i)_{i \in \N}$ be a sequence of probability measures on $A$.  A \dff{one-sided Bernoulli shift} on $A$ is the system given by the product probability space $( A^{\N}, \borel, \bigotimes_{i \in \N} \rho_i)$,  endowed with the \dff{left-shift} given by $ (T a)_i = a_{i+1}$ for all $i \in \N$, where $\borel$ is the usual Borel product 
sigma-algebra.  In the case where all the measures $\rho_i$ are identical, then we say the Bernoulli shift is an 
\dff{independent and identically distributed} (i.i.d.) system.  
\dff{Two-sided Bernoulli shifts} are similarly obtained by replacing the natural numbers with the integers, with the left-shift becoming an invertible transformation.   

 Although, Bernoulli shifts are one of the most fundamental objects in ergodic theory and probability theory, it is  difficult to exhibit type-$\mathrm{III}$ Bernoulli shifts.  Two decades after Ornstein's type-$\mathrm{III}$ odometer, Hamachi \cite{MR662470}  constructed the first nonsingular Bernoulli type-$\mathrm{III}$ systems and then three decades later, Kosloff \cite{MR2851673} constructed one that he could verify was type-$\mathrm{III}_1$.  Both of their constructions are Bernoulli shifts on two symbols, where the corresponding probabilities are defined inductively.   
See also  Vaes and Wahl \cite[Section 6]{VaesWahl}   for examples of type-$\mathrm{III}_1$ nonsingular Bernoulli shifts where the  probabilities are specified by an explicit formula. 
Type-$\mathrm{III}_1$ Markov shifts  also play a crucial role in Kosloff's  construction of type-$\mathrm{III}_1$ Anosov diffeomorphisms \cite{MR3918266,kosloff-ansov}.

In this paper, we will focus on the case $\lambda \in (0,1)$ and we will discuss briefly how we could deal with the case $\lambda =1$ in Section \ref{conclude}.
We say that a measurable map $f : \R \to [0, \infty)$ is a \dff{density} if $\int_{\R} f(u)du =1$, where $du$ represents integration with the usual Lebesgue measure.    We will identify the  density $f$ with the probability  measure given by $E \mapsto \int_E f(u) du.$

\begin{theorem} 
	\label{main-exists}
	For every $\lambda \in (0,1)$, 	there exists a choice of densities $(f_n)_{n \in \N}$ 
	such that  the Bernoulli shift $\big(\R^{\Z}, \borel, \bigotimes_{n \in \Z}  f_n\big)$ is of type-$\mathrm{III}_\lambda$, where $f_n=\mathbf{1}_{(0,1)}$ for all $n <0$.
\end{theorem}

We will also show that the (half-stationary) Bernoulli shift in Theorem \ref{main-exists} is power weakly mixing; see Section \ref{pwm} for details.

Our construction is readily adapted to yield examples in the case of a countable number of symbols, and also in  the more general setting of a countable amenable group.   We give more precise definitions in Section \ref{a-group}.

\begin{theorem}
	\label{countable}
	Let $A$ be a countable set.	For every $\lambda \in (0,1)$, there exists a Bernoulli shift on $(A^\mathbb{Z},\mathcal{B},\bigotimes_{n\in\mathbb{Z}}\rho_n)$ that is of\label{key} type-$\mathrm{III}_\lambda$.
\end{theorem}

\begin{theorem}
	\label{group-exists}
	Let $G$ be a countable amenable group and $\lambda\in(0,1)$. There exists a product measure $\bigotimes_{g\in G}f_g$ on $[0,1]^G$ such that the corresponding Bernoulli action is nonsingular, ergodic and of stable type-$\mathrm{III}_\lambda$. 
	\end{theorem}
 
In a recent article of Bj{\"o}rklund, Kosloff, and Vaes  \cite{BjoKosVaes}, they proved that in the very special case when $G$ is a  {\em locally finite group}, then for every $\lambda\in (0,1)$ there is a type-$\mathrm{III}_\lambda$ Bernoulli action on $\{0,1\}^{G}$; their  construction makes use of the locally finite assumption and when applied to  Bernoulli shifts of $\Z$, and most of the other  amenable groups, the resulting Bernoulli action is dissipative, hence not ergodic.

Let $(\Omega, \F, \mu, T)$ and $(\Omega', \F', \mu', T')$ be two nonsingular systems.  
We recall that 
a measurable map $\phi: \Omega \to \Omega'$ is a \dff{nonsingular factor} if $\phi$ is \dff{equivariant} meaning that $\phi \circ T = T' \circ \phi$, and the push-forward $\mu \circ \phi^{-1}$ is equivalent to $\mu'$; in the case where $ \mu' =    \mu \circ \phi^{-1}$, we say that $\phi$ is a \dff{measure-preserving factor}.   If $\phi: \Omega \to \Omega'$ is a factor,  then we also refer to $(\Omega', \mu', T')$ as a \dff{factor} of the original system $(\Omega, \mu, T)$.      
 Note that in the case of two measure-preserving systems, assuming ergodicity, a nonsingular factor map will also be  measure-preserving.  
       Sinai \cite{Sinai} proved under the most general possible conditions when i.i.d.\ systems are factors of measure-preserving systems.

\begin{theorem}[Sinai factor theorem]
	\label{Sinai}
A non-atomic ergodic  measure-preserving system has all  i.i.d.\ systems of no greater Kolmogorov-Sinai entropy as factors. 
		\end{theorem} 

Notice that Sinai's theorem holds  even in the non-invertible one-sided setting. Recall that  Kolmogorov-Sinai entropy  \cite{MR0103254,MR0103256} is defined for all measure-preserving systems and for an i.i.d.\ system associated with the finite probability space $(A, \rho)$ the entropy of the dynamical system given by the usual static Shannon entropy: 
$ -\sum_{c \in A} \rho(c) \log( \rho(c)).$

The Sinai factor theorem was one of the early triumphs of entropy theory which has spectacular results in identifying and classifying i.i.d.\ systems  in the measure-preserving context \cite{MR3052869} and entropy has been referred to as ``dynamical systems most glorious number'' \cite{MR2342699}.   However, the role of entropy in the general nonsingular setting remains unclear, with many results that may appear to violate intuition built from our better understanding of the measure-preserving case.  There is a  striking contrast between Krieger's finite generator theorem \cite{Krieger1} for the measure-preserving case and Krengel's generator theorem \cite{Krengel} in the nonsingular case, where the former is a landmark result in entropy theory and the latter makes no reference to entropy at all.   See \cite[Section 9]{dal-siv-up} for more information about entropy and other invariants in nonsingular dynamics.

It will become apparent that our constructions involving factors are grounded in the measure-preserving realm.   
  Given a Borel set $E \subset \R$ the \dff{conditional density of $f$ on $E$} is given by the renormalized density 
$$ \Big(\int_E f(u)du \Big)^{-1} f \mathbf{1}_E.$$

\begin{theorem}
	\label{main-fac}  Fix a Borel set $E \subset [0,1]$, and a density $g$ with support $E$.  There exists a measurable map $\phi: \R^{\N} \to [0,1]^{\N}$ such that if $(f_n)_{n \in \N}$ is a sequence of densities with the same conditional density $g$ on the set $E$, with 
\begin{equation}
\label{inf-E}
\sum_{n \in \N} \int_E f_n(u) du = \infty,
\end{equation}
and 
the 
associated Bernoulli shift $(\R^{\N}, \borel, \bigotimes_{n \in \N} f_n)$  is nonsingular,  
then $\phi$ is a measure-preserving factor from the nonsingular system  to 
the
i.i.d.\ system given by the product of Lebesgue measure restricted to the unit interval.    
		\end{theorem}

Recall that a real-valued random variable $Z$ is continuous with density $f$ if the law of $Z$ given by $\P(Z \in \cdot)$ is absolutely continuous with respect to Lebesgue measure with $f$ as its density.  Theorem \ref{main-fac} states that there exists a deterministic equivariant function $\phi$, which only depends on the set $E$ and the common density $g$, such that if $X = (X_i)_{i \in \N}$ is a sequence of continuous random variables all with the same conditional law given $E$,  then $\phi(X)$ is a sequence of independent random variables that are all uniformly distributed on the unit interval.

Theorem \ref{main-fac} and its proof also hold in the case of two-sided Bernoulli shifts.  
        The densities $f_n$ in Theorem \ref{main-exists} can be chosen to satisfy
       the conditions of Theorem \ref{main-fac}, with a uniform distribution serving as the common conditional distribution;
       see Section \ref{def-l}.   
       %
       Also recall that any real-valued random variable can be constructed as an explicit function of a uniform random variable \cite[Lemma 3.22]{kal2}.
Thus Theorem \ref{main-exists} together with Theorem \ref{main-fac} give various examples of nonsingular Bernoulli shifts which have \emph{all} i.i.d.\ systems as factors.      

\begin{corollary}
	\label{iid-fac}
	 For every $\lambda \in (0,1)$, the   type-$\mathrm{III}_\lambda$  Bernoulli shift from Theorem \ref{main-exists}  has all  i.i.d.\ factors.
		\end{corollary}

A closely related result is given by Rudolph and Silva \cite[Theorem 2.1]{rud-silva-join}, where the machinery of joinings is adapted in the nonsingular setting to construct  type-$\mathrm{III}_{\lambda}$  systems for $\lambda \in (0,1)$ as a nonsingular joinings of  two measure-preserving systems; thus it follows that i.i.d.\ factors can be obtained from these  type-$\mathrm{III}_{\lambda}$  systems, which are not given by a product measure.

Note that it is not known whether a Bernoulli shift on a \emph{finite} number of symbols can exhibit all the different types.

\begin{question}
	\label{Qfinitetypes}
	  Let $\lambda \in [0,1]$.   Does there exists a Bernoulli shift on a finite number symbols that is of type-$\mathrm{III}_{\lambda}$?
 	\end{question}

We already know that the answer to Question \ref{Qfinitetypes} is \emph{yes} for $\lambda=1$.  It turns out that   Hamachi's original example is also of type-$\mathrm{III}_1$ \cite[Section 4]{KosKMaharam}. See Section \ref{dis-rv} for more information.

We prove 
the  
following variant of Theorem \ref{main-fac} in the case of a finite number of symbols.  Whereas, entropy did not play a role in the statement of Theorem \ref{main-fac}, it will be prominent in the next theorem.  Let $A$ be a finite set, and $E \subseteq A$.  Let $\beta$ be a probability measure on the finite set $A$.  Suppose $\beta(E) >0$, then   the \dff{conditional measure of $\beta$ on $E$} is defined via $$B \mapsto  \frac{\beta(B \cap E)}{\beta(E)}.$$

\begin{theorem}[Low entropy Sinai factor]
\label{finite-fac}
Let $A$ be  a finite set.   Let $E \subset A$ have at least two elements and $\rho$ be a probability measure on $E$ with $H(\rho) >0$.  Let $\s>0$.       There exists a measurable map $\phi: A^{\N} \to \ns{0,1}^{\N}$ such that if $(p_n)_{n \in \N}$ is a sequence of probability measures on $A$ with the same conditional probability $\rho$ on  $E$,  with the  properties that
\begin{equation}
\label{inf-fE}
p_n(E) \geq \s >0  \text{ for all $n \in \N$},
\end{equation}
and that the associated Bernoulli shift  $(A^{\N}, \borel, \bigotimes_{n \in \N} p_n)$  is nonsingular, then $\phi$ is a measure-preserving factor from the nonsingular system to an i.i.d.\ system  taking two values $\ns{0,1}$.
\end{theorem}

We continue our study of factors of nonsingular Bernoulli shifts in \cite{KSFac}.

\section{Explicit constructions}
\label{def}

It will be fairly straightforward to state the densities that we will use to prove Theorem \ref{main-exists}, and we will defer the proof of Theorem \ref{main-exists} to Section \ref{proof-exists}.

\subsection{The densities for $0<\lambda<1$}
\label{def-l}
We define the densities that will be used to prove Theorem \ref{main-exists}.  
For $n \geq 2$, set  
\begin{equation}
	\label{choice}
a_n:=\frac{1}{(n+4)\log (n+4)}.
\end{equation}
 Thus $a_n$ decreases to zero with
\begin{equation}
	\label{rate}
 \sum_{n=2}^\infty a_n=\infty \  \text{ and }  \  \sum_{n=3}^\infty \left(a_{n-1}-a_n\right)<\infty.
 \end{equation}
   Let $\les(A) = |A|$ denote the Lebesgue measure or length of an interval $A$.  Let $\lambda \in (0,1)$.   Then 
   \begin{equation}
   	\label{less-one}
   \lambda a_n+a_n<1,
   \end{equation}
 for all $n \geq 2$.  
    Let $\left\{A_n\right\}_{n=2}^\infty$
    and $\left\{B_n\right\}_{n=2}^\infty$ be decreasing sequences of open intervals of $[0,1]$ satisfying:
\begin{itemize}
	\item[(a)] For all $n\in\mathbb{N}$, $A_n\cap B_n=\emptyset$.
	\item[(b)] For all $n\in\mathbb{N}$, $A_{n+1}\subset A_n$ and $B_{n+1}\subset B_n$.
	\item[(c)] For all $n\in\mathbb{N}$, $|A_n|=a_n=\lambda^{-1}\left|B_n\right|$.
\end{itemize}  
Using these sequences we define a sequence of functions $f_n:[0,1]\to \left\{{\lambda}^{-1},1,\lambda \right\}$.      For all integers $n\leq 1$, set $f_n \equiv 1$.  For  $n\geq 2$, set
\begin{equation} \label{def of f_n}
	f_n(u):=\begin{cases}
		\lambda, & u\in A_n,\\
		\frac{1}{\lambda}, &u\in B_n, \\
		1, & u\in [0,1]\setminus \left(A_n\cup B_n\right).
	\end{cases}
\end{equation}	
 For all $n\geq 2$, we have 
$$\int_0^1 f_n(u)du=\left(1-\left|A_n\right|-\left|B_n\right|\right)+\lambda\left|A_n\right|+\lambda^{-1}\left|B_n\right|=1,$$
so that the $f_n$ are densities.

 In addition, with regards to Theorem \ref{main-fac}, the strict inequality in \eqref{less-one} assures us that we may choose the set $E:= [0,1]\setminus (A_2\cup B_2)$ 
, so that conditional densities of $f_n$ on  $E$ are given by the uniform distribution.  

We will refer to Lebesgue measure 
on
$[0,1]$ as the \dff{underlying probability measure}, when comparing the construction given here to the construction given later  in Section \ref{def-c}.

\subsubsection{Brief outline of the proof of Theorem \ref{main-exists}}

We will verify that the densities defined above witness Theorem \ref{main-exists}, so that the Bernoulli shift $(\Omega, \borel, \m, T)$ is of type-$\mathrm{III}_{\lambda}$, where $\Omega = [0,1]^{\Z}$ and $\m = \bigotimes_{n \in \Z} f_n$.    A straightforward application of  Kakutani's theorem on equivalence of product measures implies nonsingularity; see Section \ref{proof-exists}.   To show that it is of the appropriate Krieger type, we employ a synthesis of  methods which were recently developed for the study of Bernoulli shifts on two symbols.
We show that the shift is conservative by using ideas appearing in Vaes and Wahl \cite[Proposition 4.1]{VaesWahl} and Danilenko, Kosloff, and Roy \cite[Proposition 2.5]{DanKosRoy}.  This is enough to imply ergodicity in the setting of a product measure, since the shift is a $K$-automorphism.

	It is well-known that the discrete Maharam extension of $T$ is ergodic if and only if 
	$T$ is of  
	type-$\mathrm{III}_{\lambda}$; see Theorem \ref{duality}.  We will show the stronger property that the Maharam extension is a conservative $K$-automorphism; we argue that the tail sigma-algebra is trivial by showing that the larger exchangeable sigma-algebra is trivial.   In the course of our proof,  we will also obtain that    
the action of the group of all finite permutations of the integers  on $\left(\Omega,\mathcal{B},\m\right)$ is ergodic.

\subsection{Discrete Random Variables}
\label{dis-rv}

Prior to Theorem \ref{main-exists}, all  constructions of type-$\mathrm{III}$ Bernoulli shifts were of type-$\mathrm{III}_1$; we already mentioned in Question \ref{Qfinitetypes} that it is not known whether Bernoulli shifts on a finite number of symbols can exhibit all the different Krieger type.   
In fact,
 under weak conditions, the behaviour of such shifts is severely limited in the following sense.   Let $A$ be a finite set.  Let $\rho_n$ be probability measures on $A$.  We say that the product measure $\mu=\bigotimes_{n \in \Z} \rho_n$ satisfies the \dff{Doeblin} condition if there exists $\delta>0$ such that for all $n\in\mathbb{Z}$ and $a\in A$, we have $\rho_n(\{a\})>\delta$. 
In the case where $A=\{0,1\}$ it was shown in \cite[Theorem B]{BjoKosVaes} that when $\mu$ satisfies the Doeblin condition, then either $\left(\{0,1\}^\mathbb{Z},\mathcal{B},\mu,T\right)$ is dissipative (see Section \ref{conserve}) or it is ergodic and of type $\mathrm{II}_1$ or $\mathrm{III}_1$. Recently,  Avraham-Re'em \cite{Nachi} extended this dichotomy to the broader class of inhomogeneous Markov shifts supported on mixing subshifts of finite type; a particular case of this result is the following. 
 \begin{theorem}[Avraham-Re'em]  
 	\label{reem}
 	Let $A$ be a finite set and    $\left(A^\mathbb{Z},\mathcal{B},\mu,T\right)$ be a nonsingular Bernoulli shift.  If the product measure $\mu$ satisfies the Doeblin condition, then the system 
 	is   either  
 	 dissipative or it is ergodic and of type $\mathrm{II}_1$ or $\mathrm{III}_1$. 
 \end{theorem}
 The question arises whether Theorem \ref{reem} holds when the product measure does not satisfy the Doeblin condition, as in the countable case.  When $A=\{0,1\}$ the following question arising from \cite{BjoKosVaes} is still open. 
 \begin{question}
 	\label{possible-nonD}
 	What are the possible Krieger types of the Bernoulli shift  
 	$\left(\{0,1\}^\mathbb{Z},\mathcal{B},\bigotimes_{n\in\mathbb{Z}}\rho_n,T\right)$ with  $\lim_{|n|\to\infty}\rho_n(0)=0$?
 \end{question} 

\subsubsection{The probability mass function for the countable case}
\label{def-c}
We will adapt the our construction in Section \ref{def-l} to the countable setting by  replacing the decreasing sequences of subsets of the unit interval with decreasing sequences of subsets of $\N$.   

Let $0<\lambda<1$.   For each integer $n \geq 1$, let  $$a_n=\frac{1}{(n+4)\log(n+4)}.$$ Let $\rho$ be the probability mass function on $\bb{N}$ defined by
$$
\rho(n)=\begin{cases}
1-(1+\lambda)a_1, &\ n=0,\\
a_k-a_{k+1}, &\ n=2k,\\
\lambda\left(a_k-a_{k+1}\right),&\ n=2k+1.
\end{cases}
$$
We will refer to $\rho$ as the \dff{underlying probability measure}.  For each integer $n \geq 1$, set  $$A_n=2\bb{N}\cap [2n,\infty) \text{ and  }  B_n=\left(\bb{N}\setminus 2\bb{N}\right)\cap [2n-1,\infty),$$ then
\[
\rho\left(A_n\right)=\sum_{k=n}^\infty \left(a_k-a_{k+1}\right)=a_n=\lambda^{-1}\rho\left(B_n\right). 
\]
For each integer $n \geq 1$,  let $f_n:\bb{N}\to \{\lambda^{-1},1,\lambda\}$ be defined via
\begin{equation} \label{def of f_n, countable}
f_n(k)=\begin{cases}
\lambda, & k\in A_n,\\
\frac{1}{\lambda}, &k\in B_n, \\
1, & k\in \bb{N}\setminus \left(A_n\cup B_n\right).
\end{cases}
\end{equation}	
We collect the following useful observation for future reference.
\begin{remark}
	\label{int-den-p}
Note that for all $n \geq 1$, we have
\[
\sum_{k=1}^\infty f_n(k)\rho(k)=\rho\left(\bb{N}\setminus \left(A_n\cup B_n\right)\right)+\lambda \rho\left(A_n\right)+\frac{1}{\lambda}\rho\left(B_n\right)=1.
\]
Thus  the functions $f_n$ are  probability density functions with respect the underlying probability measure $\rho$. \erk 
\end{remark}
Finally, define the product measure $p$ on $\bb{N}^\bb{Z}$ by
$$
p_n=\begin{cases}
\rho, &\ n\leq 0,\\
f_n \rho,& \ n\geq 1.
\end{cases}
$$

Let $\Omega = \N^{Z}$, and $\borel$ be the usual product sigma-algebra, and $\mu = \bigotimes_{n \in \Z} p_n$.  We will show that the Bernoulli shift $(\Omega, \borel, \mu, T)$ witnesses Theorem \ref{countable}.  After we prove Theorem \ref{main-exists}, we will see that the proof of  Theorem \ref{countable}, given in Section \ref{conclude}, will be a straightforward adaptation of the proof for the continuous random variables.

\section{The proofs of Theorems \ref{main-fac} and \ref{finite-fac}}

In our  proof of Theorem \ref{main-fac}, we will harness independent uniform random variables for every integer $n$ for which $x_n \in E \subset [0,1]$; these uniform random variables will then be distributed to the other integers.   Kalikow and Weiss \cite{MR1143427}  elegantly use similar ideas  to construct explicit 
isomorphisms 
of some infinite entropy processes.

\begin{proof}[Proof of Theorem \ref{main-fac}]
	Let the conditional density of the $f_n$ on $E$ be $g$. Thus if $G: \R \to [0,1]$ is the cumulative distribution function given by
	$$ G(v) = \int_{-\infty} ^v g(u) du$$
	we easily verify that if $V$ is a random variable with density $g$, then $G(V)$ is uniformly distributed in $[0,1]$.  We also note  that   by taking binary expansions, it is easy to see that there exists a measurable function $r : [0,1] \to [0,1]^{\N}$ such that if $U$ is uniformly distributed in $[0,1]$, then  $r(U)$ is 
	 an 
	 i.i.d.\ sequence of random variables that are uniformly distributed in $[0,1]$; for details see \cite[Lemma 3.21]{kal2}.
	
	Let $X = (X_n)_{n \in \N}$ 
	be a sequence of independent continuous random variables with corresponding densities $(f_n)_{n \in \N}$.   Call $s \in \N$ \dff{special} if $X_s \in E$; thus conditional on the event that $s$ is special, $G(X_s)$ is uniformly distributed in $[0,1]$; furthermore, conditional on the sequence of special integers $s_n$, the sequences of random variables $(r\circ G)(X_{s_n})_{n \in \N}$ are independent.  
	
	Thus the assumption on the densities $f_n$ make  it is easy to independently assign a uniform random variable to each special natural number in an equivariant way.    
Observe that by 
\eqref{inf-E} and the second Borel-Cantelli lemma there are infinitely many special natural numbers.    It remains to independently assign each non-special natural number a uniform random variable, in an equivariant way.  
	
	We say that each $k \in \N$ \dff{reports} to the smallest integer greater than or equal to $k$ that is special; thus if $s$ is special, then it reports to itself.  If $k \in \N$ reports to the special integer $s$, we set
	$$ [\phi(X)]_k =   [r(G(X_s))]_{s-k}.$$   
	It is easy to verify that $\phi$ satisfies the required properties.
	\end{proof}

Our proof of Theorem \ref{finite-fac} is slightly more involved than our proof of Theorem \ref{main-fac}, since in the finite entropy regime we cannot replicate uniform random variables.  At the special integers we only get i.i.d.\ discrete random variables; these random variables will be transformed using Theorem \ref{Sinai} into \dff{bits}, that is, zero and ones, that will then be distributed using the following equivariant matching scheme.

  Consider $d \in \Z^{+}$ and  a subset  $\Omega' \subset  \ns{a,b}^{\N}$  with $T(\Omega') \subset \Omega'$, where $T$ is the left-shift.     We want to define a loop-free graph $G(\omega)$  on $\N$ with the following properties.
\begin{itemize}
\item
If $m$ and $n$ are adjacent and $m <n $, then $\omega_m = b$ and $\omega_n =a$.
\item
Each vertex $m$ with $\omega_m=b$ is of degree $1$.
\item
Each vertex $n$ such that $\omega_n=a$ has degree at most $d$.
\item
If $n,m \geq 1$, then  the vertices $n$ and $m$  are adjacent in $G(\omega)$ if and only if  $n-1$ and $m-1$ are adjacent in  $G(T \omega)$.
\end{itemize}
 We call $G$ a \dff{degree-$d$ equivariant matching scheme}.  Thus $G$ is a matching of $a$'s and $b$'s, where every $b$ is matched to a unique $a$, and each $a$ has at most $d$ partners.

\begin{proposition}  
	\label{mel}
Consider the Bernoulli shift  $( \ns{a,b}^{\N}, \borel, \bigotimes_{n \in \N} p_n)$, where $p_n(a) \geq \s >0$ for all $n \in \N$.    Let $d \in \Z^{+}$ be such that $d \geq (1-\s)/\s$.   There exists a degree-$d$ equivariant matching scheme $G$ on a set of full measure.  
\end{proposition}

 Me{\v{s}}alkin \cite{MR0110782} gave the first example of a nontrivial isomorphism between two i.i.d.\ systems and almost a half a century afterwards, 
 his map was adapted by Holroyd and Peres \cite{MR2118858}  to define a perfect equivariant matching scheme for the case of i.i.d.\ fair coin-flips.     Our proof of Proposition \ref{mel} uses a similar adaptation of  Me{\v{s}}alkin  map.   For related matchings constructions in probability theory see also \cite{MR2264945,  random, Sa}.

\begin{proof}[Proof of Proposition \ref{mel}]
Let $\omega \in \ns{a,b}^{\N}$.  We define the graph $G(\omega)$ inductively in the following way.  For each $n \in \N$, if  $\omega_n = b$ and $\omega_{n+1} = a$, then add the edge $\ns{n, n+1}$ to the graph; that is, we match an $a$ and $b$ if $b$ is immediately followed by an $a$.  Now we disregard all the $b$'s that have been matched, and all the $a$'s that are already of degree $d$, and repeat inductively.   

Note that by definition if $n < m$, and $\omega_{n}=b = \omega_m$, then if $(n, k)$ is an edge with $k >m$, then $(m, \ell)$ is an edge for some $ m < \ell \leq k$.

To show that every $b$ is matched, let $(X_n)_{n \in \N}$ be independent $\ns{-1, d}$ valued random variables with  $\P(X_n =d) =p_n(a)$.  Identify $a$'s with $d$'s and $b$'s and with $-1$'s.  Fix $m \in \N$.    Suppose $X_m=-1$.   Let $$S_n := X_m + \cdots + X_{m+n}$$ and   let $$R:= \inf\ns{k \geq 1: S_{k} \geq 0}.$$   Then  
\begin{equation}
	\label{event}
	\P\ns{ \text{$(m,m+\ell)$ is not an edge for all $\ell \leq k$}} = \P({R >k}).
\end{equation}
It remains to  show that the right hand side of \eqref{event} decays to zero as $k \to \infty$.     Let $(X_n')_{n \in \N}$ be i.i.d.\ $\ns{-1, d}$-valued random variables with $\P(X_0'=d) =  \delta$.   Also assume that $X'_m=-1$ and     
 similarly define $S'_n$ and $R'$.  A standard coupling argument gives that
 $$ \P(R >k) \leq \P(R' >k).$$  
 Since $$\mathbb{E} X_0'   = d\s -(1-\s) \geq 0,$$
if $\E X_0' > 0$, then the law of large numbers gives that $R'$ is finite almost surely and if $\E X_0' =0$, then classical results of Chung and Ornstein \cite{MR40610,MR133148} regarding the recurrence random walks  imply that $R'$ is finite almost surely.     
	\end{proof}

\begin{proof}[Proof of Theorem \ref{finite-fac}]
Let $d \in \Z^{+}$ be so that $d \geq (1-\s)/\s$. Let $\alpha$ be a probability measure on $\ns{0,1}$ such that
\begin{equation*}
 (d+1)H(\alpha) \leq H(\rho).
 \end{equation*}
By Theorem \ref{Sinai}, it follows there exists a factor $\psi$ from the i.i.d.\ system with common distribution $\rho$ to the i.i.d.\ system with common probability measure $\alpha^{d+1}$ on $\ns{0,1}^{d+1}$.   

Let $X = (X_i)_{i \in \N}$ be a sequence of independent  random variables with corresponding probability mass functions $(p_i)_{i \in \N}$.       Call $s \in \N$ \dff{special} if $X_s \in E$; thus conditional on the event that $s$ is special, $(X_s)$ has p.m.f.\ $\rho$;  furthermore, conditional on the sequence of special integers $s_n$, the sequence of random variables $Y=(X_{s_n})_{n \in \N}$ are independent.  We apply the factor map $\psi$ on $Y$ to obtain at each special integer, $d$-independent bits with distribution $\alpha$.  

Similarly to the 
proof of Theorem \ref{main-fac},  
it remains  to distribute the bits from  the special integers, in an equivariant way, so that each integer has a bit; this can be  accomplished via Proposition \ref{mel}.  Each special integer retains one bit and allows the remaining $d$ bits to be distributed to the non-special integers according to the given equivariant matching; any remaining bits are discarded.
\end{proof}

\begin{remark}
At the outset, we invoked Sinai's factor theorem in our proof of Theorem \ref{finite-fac}.  If one required a more constructive proof, and an explicit factor map, we could instead appeal to del Junco's finitary version of Sinai's theorem \cite{Juncoa, Junco}. However, then we must assume that the set $E$ contains at least three symbols and require also that resulting  Bernoulli shift be on three symbols.  Keane and Smorodinsky's celebrated finitary isomorphisms \cite{keanea, keaneb} do not have a symbol restriction, but are not one-sided. \erk
	\end{remark}

 \section{Krieger's ratio set}
 \label{krs}

   In what follows, it will be convenient to think of an invertible nonsingular dynamical systems  $T= (T^n)_{n \in \Z}$ as the integer group action.     
Let $G$ be a group and write $\mu \sim \nu$ for two equivalent measures. A \dff{nonsingular group action} is a measure space $(\Omega, \F, \mu)$  endowed with a group action $T=(T_g)_{g \in G}$ such that $T_g \circ T_h = T_{gh}$ and $\mu \circ T_g \sim \mu$ for all $g,h \in G$; we say it is \dff{ergodic}  
if $E\in\B$ satisfies for all $g\in G$,  $\mu(E \triangle T_g(E)) =0$, then either 
$\mu(E)=0$ 
or $\mu(E^c) =0$, 
for all $E \in \F$.  
We say that $T$ is \dff{conservative} if for every $A  \in \F$ with positive measure, there 
exists  $g \in G$
that is not the 
unit of $G$, 
with $\mu(A \cap T^{-g}A) >0$, and otherwise we say that $T$ is \dff{dissipative}.  We will often identify $T_g$ with 
 $g$.

We say that two nonsingular group actions $(\Omega, \F, \mu, (T_g)_{g \in G})$ and $(\Omega', \mathcal{G}, \nu, (S_h)_{h \in H})$ are \dff{orbit equivalent} if there exists  a measurable bijection $\phi: \Omega \to \Omega'$ such that 
$\nu \sim \mu \circ \phi^{-1} $ and $\phi( \orb_G(x)) := \orb_H(\phi(x))$ for $\mu$-almost all $x \in \Omega$, where $\orb_G(x) = \ns{ T_g(x):  g \in G}$.   Motivated by problems in  von Neumann algebras,  Dye \cite{MR131516,MR158048} proved in the setting of a probability preserving system that any two Abelian discrete 
ergodic
group actions on non-atomic measure spaces are orbit equivalent.    

Following the work of Araki and Woods \cite{MR0244773}, who were again motivated by the Murray-von Neumann classification problem, Krieger \cite{kriegertypes,MR415341} extended Dye's celebrated result to the nonsingular setting.  For an ergodic nonsingular action  a number $r\in\mathbb{R}$ is an \dff{essential value} for $T$ if for all $A\in\F$ with $\mu(A)>0$ and $\epsilon>0$ there exists $g\in G$ such that 
 \[
 \mu\left(A\cap T_g^{-1}A\cap \bigg[\Big|\log\frac{d\mu\circ T_g}{d\mu}-r\Big|<\epsilon\bigg]\right)>0. 
 \]
 The \dff{Krieger ratio set} $e(T)$ is the collection of all essential values of $T$. The ratio set is a closed subgroup of $\mathbb{R}$ hence it is of
 \begin{itemize}
 \item
   type-$\mathrm{II}$ or type-$\mathrm{III}_0$:   $e(T) = \{0\}$;
  \item
 type $\mathrm{III}_\lambda$:   $e(T) = \{n\log\lambda :n\in\mathbb{Z} \}$ for some $0<\lambda<1$; or 
  \item
  type $\mathrm{III}_1$: $e(T) = \R$.
  \end{itemize}
  
The Krieger types are invariants for orbit equivalence and are a complete invariant  when $e(T)$ is  nonempty and $e(T) \not = \ns{0}$; this classification holds for any discrete amenable group action.   See  \cite{MR662736} for background and more details.

  We will use the following lemma to verify whether a given number is an essential value for $T$.  
 The \dff{orbital equivalence relation} of the action $T$ is the Borel subset $\mathcal{O}_T\subset X\times X$ defined by 
 \[
 (x,y)\in \mathcal{O}_T\ \  \text{if and only if there exists } g\in G \text{ such that }  T_gx=y.  
 \]
The \dff{full group} $[T]$ consists of all nonsingular automorphisms $V$ of $\left(\Omega,\F,\mu\right)$ such that for almost all $x \in \Omega$, we have $(x,Vx)\in\mathcal{O}_T$.    Let $A \in \F$.   We say that an injective nonsingular map $V:A\to V(A)\in\mathcal{F}$ such that $(x,Vx) \in \mathcal{O}_T$ for all $x \in A$ is a \dff{partial transformation} with domain $A$ and range $V(A)$.   The collection of partial transformations will be denoted by $[[T]]$.

 Recall that a collection of subsets of $\F$ is \dff{$\mu$-dense} if  for every $\e >0$ and every $F \in \F$ there exists a $F'$ from the collection such that $\mu(F' \triangle F) < \e$.  

 \begin{lemma}[Approximation]
 	\label{lem: CHP}
 Let $(\Omega, \F, \mu, T)$ be a nonsingular,  ergodic action of a 
 countable group. 	Let $\mathcal{G}\subset\mathcal{F}$ a countable semi-ring such that the ring generated by $\mathcal{G}$ is $\mu$-dense in $\mathcal{F}$.  If there exists $0<\delta<1$ such that for each $A\in \mathcal{G}$ and $\epsilon>0$ there 
 are:
 	\begin{itemize}
 		\item a subset $B\subset A$ with $\mu(V(B))>\delta\mu(A)$ and
 		\item a partial transformation $V:B\to A$ such that $(x,Vx)\in\mathcal{O}_T$ and for all $x\in B$, $\left|\frac{d\mu\circ V}{d\mu}(x)-r\right|<\epsilon$,
 	\end{itemize}
 	then $r\in e(T)$. 	  
 \end{lemma}
 \begin{proof}
 The proof is a routine extension of the second half of \cite[Lemma 2.1]{ChHaPr87}, where we allow $V \in [[T]]$, rather than requiring that $V \in [T]$.    See also \cite[Lemma 2.1]{DaniLem}.
 	\end{proof}

 \section{Proof of Theorem \ref{main-exists}}
 \label{proof-exists}
 
 \subsection{Nonsingularity and Kakutani's theorem}

 Let $(\Omega, \F, \mu, T)$ 
 be an invertible   nonsingular system. 
      We write $$T':=\frac{d\mu\circ T}{d\mu}.$$
  Let $\Omega=[0,1]^\mathbb{Z}$.  Given a sequence of densities $f_n:[0,1]\to (0,\infty)$, let $\mu = \bigotimes_{n \in \Z} f_n$ be the associated product measure.   Since $\m\circ T=\bigotimes_{n\in\mathbb{Z}}f_{n-1}$ is also a product measure, it follows from Kakutani's theorem \cite{MR23331} on equivalence of product measures that $\m$ is $T$ nonsingular if and only if
  \begin{equation}\label{eq: Kakutani}
  	\sum_{n\in\mathbb{Z}} \int_0^1 \left(\sqrt{f_n(u)}-\sqrt{f_{n-1}(u)}\right)^2du<\infty. 
  \end{equation}
 
 With the densities $f_n$ defined in Section \ref{def}, by Kakutani's theorem, for $\m$-almost every $x\in \Omega$ and  for all $n\in\mathbb{Z}$, we have
 \begin{equation}
 \label{lambda}
  \left(T^n\right)'(x):=\frac{d\m\circ T^n}{d\m}(x)=\prod_{k\in\mathbb{Z}}\frac{f_{k-n}\left(x_k\right)}{f_k\left(x_k\right)}\in\left\{\lambda^n: n\in\Z\right\}. 
  \end{equation}
  In particular, we have $e(T)\subset\left\{\lambda^n:\ n\in\mathbb{Z}\right\}$ and thus to show that $T$ is type $\mathrm{III}_\lambda$ it is enough to show that $T$ is ergodic and that $\lambda\in e(T)$.

\begin{lemma}
	\label{nonsing}
	 With the densities $f_n$ defined in Section \ref{def},  the associated  Bernoulli shift is nonsingular. 
		\end{lemma}
\begin{proof}
  
Setting $A_1=B_1=\emptyset$, we see that
 \begin{align*}
 	\sum_{n\in\mathbb{Z}} \int_0^1 \Big(\sqrt{f_n(u)}-\sqrt{f_{n-1}(u)}\Big)^2du&=\sum_{n=2}^\infty \int_0^1 \Big(\sqrt{f_n(u)}-\sqrt{f_{n-1}(u)}\Big)^2du\\
 	&=\sum_{n=2}^\infty \big(\lambda \left|A_{n-1}\setminus A_n\right|+\frac{1}{\lambda}\left|B_{n-1}\setminus B_n\right|\big)\\
 	&\leq\frac{\lambda}{6\log 6}+\sum_{n=3}^\infty 2\lambda\left(a_{n-1}-a_n\right).
 \end{align*}
 The finiteness of the right-hand side follows from \eqref{rate} and thus Kakutani theorem implies the desired nonsingularity.   
 \end{proof}
 
 \subsection{Conservativity and ergodicity}
 \label{conserve}
 
 Let $(\Omega, \F, \mu)$ be a measure space.   Let $T: \Omega \to \Omega$ be a nonsingular transformation.  
It is well-known that if $T$ is ergodic and $\mu$ is non-atomic, then $T$ is conservative, whereas the converse fails.   However, there is a partial converse, in the case that $T$ is an $K$-automorphism in the sense of the Kolmogorov zero-one law.  Suppose that $T$ is invertible.     Following Silva and Thieullen \cite[Definition 4.5]{SilThe95}, we say that $T$ is a \dff{$K$-automorphism}  if there exists
a sigma-algebra 
 $\mathcal{G}\subset \F$ such that:
  \begin{itemize}
  	\item $\mu|_\mathcal{G}$ is $\sigma$-finite and; $T^{-1}\mathcal{G}\subset \mathcal{G}$; 
  	\item $\bigcap_{n\in\mathbb{N}}T^{-n}\mathcal{G}=\{\emptyset,X\} \bmod \mu$;
  	\item  $\bigvee_{n\in\mathbb{N}}T^n \mathcal{G}=\mathcal{F} \bmod \mu$; 
  	\item $T'$ is $\mathcal{F}$ measurable.
  \end{itemize}
  The first three conditions are that $T$ is a natural extension of an endomorphism with a trivial tail field
  while the fourth comes to ensure that the natural extension is unique up to measure theoretic isomorphism of nonsingular systems. 
  
  We will make use of the following proposition from Silva and Thieullen \cite[Proposition 4.8]{SilThe95}; see also Parry \cite{ParryKauto65}.
  
  \begin{lemma}[Silva and Thieullen]
  	\label{ST-K}
  A $K$-automorphism is ergodic if and only if it is conservative.
  \end{lemma}

 \begin{lemma}
 	\label{lem: VW}
 	Let $\lambda \in (0,1)$.  Let $f_n$ be the densities defined in Section \ref{def}, let $\m = \otimes_{n \in \Z} f_n$ and consider the associated Bernoulli shift.   
 	Then with 
 	$$c(\lambda) = 2(\lambda^3-1+\lambda^{-2}-\lambda),$$  for all $n\in\mathbb{N}$, we have	
 	\[
 	\int_{\Omega } \left(\frac{1}{\left(T^n\right)'}\right)^2d\mu\leq 
 	\exp\left(c(\lambda)\sum_{k=2}^{n+1}a_k\right). 
 	\]
 \end{lemma}
 \begin{proof}
 	Let $n\in\mathbb{N}$. Since $\m$ is a product measure and 
 	$f_k=f_{k-n} \equiv 1$ for all 
 	$k \leq  1$, we have
 	\[
 	\int_{\Omega } \left(\frac{1}{\left(T^n\right)'(x)}\right)^2d\mu(x)=
 	\prod_{k=1}^\infty 
 	\int_0^1 \left(\frac{f_{k}(u)}{f_{k-n}(u)}\right)^2f_k(u) du.
 	\]
 	First note that for $2\leq k\leq n+1$,
 	we have 
 	$f_{k-n}\equiv 1$ and thus,
 	\begin{align*}
 		\int_0^1\left(\frac{f_{k}(u)}{f_{k-n}(u)}\right)^2f_k(u)du&= \int_0^1f_k(u)^3du\\
 		&=\int_0^1\left(\lambda^3\1_{A_k}+\lambda^{-3}\1_{B_k}+\1_{(A_k\cup B_k)^c}\right)du\\
 		&=1+\left(\lambda^3-1\right)\left|A_k\right|+\left(\lambda^{-3}-1\right)\left|B_k\right|\\
 		&=1+\left(\lambda^3-1+\lambda^{-2}-\lambda\right)\left|A_k\right|\\
 		&\leq \exp\left(\frac{c(\lambda)}{2}a_k\right).
 	\end{align*}
 	For every $k\geq n+2$, since $B_k\subset B_{k-n}$ and $A_k\subset A_{n-k}$ we see that,
 	\[
 	\left(\frac{f_{k}}{f_{k-n}}\right)^2f_k=\lambda \1_{A_k}+\lambda^{-2}\1_{A_{k-n}\setminus A_k}+\lambda^{-1}\1_{B_k}+\lambda^2\1_{B_{n-k}\setminus B_k}+\1_{ \left(A_{k-n}\cup B_{k-n}\right)^c}.
 	\]
 	Adding and subtracting $\lambda \1_{A_{k-n}\setminus A_k}+\lambda^{-1}\1_{B_{k-n}\setminus B_k}$ to the right hand side shows that 
 	\begin{align*}
 		\left(\frac{f_{k}}{f_{k-n}}\right)^2f_k&=
 		\Big[\big(\lambda \1_{A_{k-n}}+\lambda^{-1}\1_{B_{n-k}}+\1_{\left(A_{k-n}\cup B_{k-n}\right)^c}\big)+\left(\lambda^{-2}-\lambda \right)\1_{A_{k-n}\setminus A_k} \\
 		&\ \ \ +\left(\lambda^2-\lambda^{-1}\right)\1_{B_{k-n}\setminus B_k} \Big]\\
 		&=f_{k-n}+\left(\lambda^{-2}-\lambda \right)\1_{A_{k-n}\setminus A_k}+\left(\lambda^2-\lambda^{-1}\right)\1_{B_{k-n}\setminus B_k}
 	\end{align*}
 	Integrating we see that for all $k\geq n+2$, 
 	\begin{align*}
 		\int_0^1 \left(\frac{f_{k}(u)}{f_{k-n}(u)}\right)^2f_k(u) du&=
 	\bigg[ \int_0^1 f_{k-n}(u) du+\left(\lambda^{-2}-\lambda \right) \left|A_{k-n}\setminus A_k\right| \\
 		&\ \ \ \ +\left(\lambda^2-\lambda^{-1}\right)\left|B_{k-n}\setminus B_k\right| \bigg]   \\
 		&=1+\left(\lambda^{-2}-\lambda+\lambda\left(\lambda^2-\lambda^{-1}\right)\right)\left(\left|A_{k-n}\right|-\left|A_k\right|\right)\\
 		&\leq \exp\left(\frac{c(\lambda)}{2}\left(a_{k-n}-a_{k}\right)\right).
 	\end{align*}
 	Combining the inequalities we have obtained, we see that
 	\begin{align*}
 		\int_{\Omega } \left(\frac{1}{\left(T^n\right)'}\right)^2d\m&\leq \exp\left(\frac{c(\lambda)}{2}\left(\sum_{k=2}^{n+1}a_k+\sum_{k=n+2}^\infty \left(a_{k-n}-a_k\right)\right)\right)\\
 		&=\exp\left(c(\lambda)\sum_{k=2}^{n+1}a_k\right).
 		\qedhere
 	\end{align*}

 \end{proof}	
 
 \begin{lemma}
 	\label{cons}
 	The Bernoulli shift associated with the densities defined in Section \ref{def} is conservative. 
 \end{lemma}
 \begin{proof}
 	By \eqref{choice} 
 	$$\sum_{k=2}^n a_k=\log\log(n)+O(1).$$  By Lemma \ref{lem: VW},  there exists 
 	$K>1$ 
 	such that for all $n\in\mathbb{N}$,
 	\begin{equation}\label{eq: VW}
 		\int_{\Omega } \left(\frac{1}{\left(T^n\right)'}\right)^2d\m\leq 
 		K
 		\left(\log(n+1)\right)^{c(\lambda)}. 
 	\end{equation}
 	Markov's inequality give that
 	\begin{align*}
 	\m\left\{ x\in\Omega: (T^n)'(x)^{-2}>n^{2}\right\}
 		&\leq 
 		K
 		\frac{\left(\log(n)\right)^{c(\lambda)}}{n^2}.
 	\end{align*}
 	Since 
 	$$	\m\big\{ x\in\Omega: \left(T^n\right)'(x)<n^{-1}\big\} =  	\m\left\{ x\in\Omega: (T^n)'(x)^{-2}>n^{2}\right\},$$
 	by the first Borel-Cantelli lemma, for almost every $x\in\Omega$ there exists $n(x)\in\mathbb{N}$ such that for all $n\geq n(x)$, we have $\left(T^n\right)'(x)\geq \frac{1}{n}$.   Since the harmonic series diverges,  the comparison test gives that for almost every $x\in\Omega$,  we have
 	\[
 	\sum_{n=1}^\infty\left(T^n\right)'(x)=\infty.
 	\]
It follows from the Hopf criteria \cite[Proposition 1.3.1]{AaroBook} that $T$ is conservative.
 \end{proof}
 
 \begin{corollary}
 	\label{together}
 	The Bernoulli shift associated with the densities defined in Section \ref{def} is nonsingular, conservative, and ergodic. 
 	\end{corollary}
 
 \begin{proof}
We already know the associated Bernoulli shift is nonsingular from   Lemma \ref{nonsing}.  	Product measures satisfy 
Kolmogorov's 
zero-one law \cite[Appendix]{MR0079843} and thus all 
half-stationary  
nonsingular Bernoulli shifts are $K$-automorphisms.  The result follows from Lemma \ref{cons}  and Lemma \ref{ST-K}.
 	\end{proof}

 \subsection{Maharam extensions}
 	\label{ME-Z}

 	A useful duality in studying the  ratio set is given by the following skew product.  The \dff{Maharam extension} of a nonsingular group action $\big(\Omega,\mathcal{\F},\mu,\left(T_g\right)_{g\in G}\big)$ is a $G$ action on $\Omega \times \mathbb{R}$, given by
 	\[
 	\tilde{T}_g(x,u)=\Big(T_gx,u-\log\frac{d\mu\circ T_g}{d\mu}(x)\Big),
 	\]
 	which preserves the measure given by $$\nu\left(A\times I\right):=\mu(A)\int_I e^{u}du,$$ 
 	for all $A\in\mathcal{F}$ and intervals $I\subset \mathbb{R}$.  
 	By the celebrated result of  Maharam \cite{MR169988}, in the case where $G= \Z$, the $\mathbb{Z}$-action $\left(\Omega,\mathcal{F},\mu,T\right)$ is conservative if and only if its Maharam extension is conservative with respect to $\nu$.

Let $\lambda \in (0,1)$.  Consider the nonsingular system $(\Omega, \F, \mu, T)$.   Suppose that for $\mu$-almost every $x\in \Omega$, we have  
\begin{equation}
	\label{log-lambda}
\varphi_\mu(x):=\log_\lambda\left(\frac{d\mu\circ T}{d\mu}(x)\right)\in\mathbb{Z}.
\end{equation}
   Then we define the \dff{discrete Maharam extension} on $\Omega \times \mathbb{Z}$  by
 	\begin{equation}
 	\tilde{T}(x,n):=\left(Tx,n-\varphi_\mu(x)\right).	
 	\end{equation}
 	The discrete Maharam extension preserves the measure  $\tilde{\mu}$ such that for all $A\in\F$ and $n\in\bb{Z}$, we have $\tilde{\mu}(A\times\{n\})=\lambda^{n}\mu(A)$. 
 	More generally, in the context of a group action $(T_g)_{g \in G}$ if 
\begin{equation}
 \label{log-lambda2}
\varphi_\mu(x,g):= \log_\lambda\left(\frac{d\mu\circ T_g}{d\mu}(x)\right)\in\mathbb{Z}
 \end{equation}
 for all $g \in G$, then we define the \dff{discrete Maharam extension of the $G$-action} by 
 	$$ 	\tilde{T}_g(x,n):=\left(T_gx,n-\varphi_\mu(x,g)\right).$$

 	\begin{theorem} 
 		\label{duality}
 		 Let $\lambda \in (0,1)$.  Let $(\Omega, \F, \mu, (T_g)_{g \in G})$ be a nonsingular group action such that \eqref{log-lambda2} holds.  Then the nonsingular system is conservative if and only if the discrete  Maharam extension of the $G$-action  is conservative and furthermore the nonsingular system is  of type-$\mathrm{III}_\lambda$ if and only if the discrete  Maharam extension is ergodic.  
 		\end{theorem}
 
 \begin{proof}
 	A proof of this theorem, in a more general setting, is given in the monograph of Klaus Schmidt \cite[Corollary 5.4, Theorem 5.5]{KS-type}.   
 	\end{proof}

 Note that by \eqref{lambda} the Bernoulli shift defined in Section \ref{def} 
 satisfies  \eqref{log-lambda}.

 \begin{theorem}
 	\label{duality-checked}
 		 The discrete Maharam extension associated with the  nonsingular Bernoulli shift defined in Section \ref{def} is a $K$-automorphism.
 	 	\end{theorem}

Before we prove Theorem \ref{duality-checked}, we first show,
how it leads to a proof of Theorem \ref{main-exists}.

\begin{proof}[Proof of Theorem \ref{main-exists}]
Let $\lambda \in (0,1)$.  	Consider the  Bernoulli shift defined in Section \ref{def}.   We already know it is nonsingular, conservative, and ergodic by Corollary \ref{together}.   By Theorem \ref{duality} its Maharam extension is also conservative.   By Theorem \ref{duality-checked}, the associated discrete Maharam extension is a $K$-automorphism and thus by Lemma \ref{ST-K} it is ergodic.  Thus by Theorem \ref{duality} the nonsingular Bernoulli shift is of type-$\mathrm{III}_{\lambda}$.
\end{proof} 
 
 It remains to prove Theorem \ref{duality-checked}; 
 it will be given in  Section  \ref{perm}.   	
 We will use {\em both} directions of Theorem \ref{duality} in our proof of Theorem \ref{duality-checked}.     Let $\lambda \in (0,1)$.  Consider the Bernoulli shift defined in Section \ref{def} and its one-sided version defined as follows. Let  $\B_+ = \B([0,1]^{\N})$  be the product sigma-algebra for $\Omega_{+}=[0,1]^{\N}$, $\mu_+:=\bigotimes_{n\in\mathbb{N}}f_n$, and $S$ be $T$ restricted to $[0,1]^{\N}$,  then $\left([0,1]^\mathbb{Z},\B,\mu,T\right)$ is the natural extension of 
 $\left([0,1]^\mathbb{N},\B_+,\mu,S\right)$ and the latter has a trivial tail field, by  Kolmogorov's zero-one law.

    Recall that for the {\em two}-sided system, we defined 
 $$	\varphi_\mu(x):=\log_\lambda\left(\frac{d\mu\circ T}{d\mu}(x)\right) = \log_{\lambda}  \Big( \prod_{k\in\mathbb{N}}\frac{f_{k-1}\left(x_k\right)}{f_k\left(x_k\right)} \Big)\in\mathbb{Z}.$$
  Note that the \dff{restriction} of $(T^n)'$ to $\Omega_{+}$ is  given by
  $$ \prod_{k\in\mathbb{N}}\frac{f_{k-n}\left(x_k\right)}{f_k\left(x_k\right)}$$
 is 
 $\B_{+}$-measurable; 
 in a slight abuse of notation we will also continue to denote the restriction of $\varphi_\mu$ to $\Omega_{+}$ by $\varphi_\mu$. 
 The discrete Maharam extension of $T$ is the natural extension of the skew product extension of $S$ by the restriction of  $\varphi_{\mu}$, defined by  on $\Omega_{+} \times \Z$ given by
 \[
 S_{\varphi_\mu}(x,n):=(Sx,n-\varphi_\mu(x)).
 \]
 Note that $S_{\varphi_\mu}$ preserves the measure $\widetilde{\mu_+}$ which is the restriction of $\tilde{\mu}$ to $\Omega_{+} \times \Z$.    Therefore in order to prove that $\tilde{T}$ is a $K$-automorphism, it suffices to show that the tail 
 sigma-field 
 of $S_{\varphi_\mu}$ is trivial.

It is well-known that the Hewitt-Savage zero-one law \cite{MR76206} implies Kolmogorov's zero-one law.  Similarly, we will prove that the exchangeable sigma-field of $S_{\varphi_\mu}$ is trivial. As the tail field is a subset of the exchangeable 
sigma-field, 
this will establish that the tail field is trivial and consequently that $\tilde{T}$ is a $K$-automorphism.

 \subsection{The ergodic action of the  permutation group}
 \label{perm}
 
 We say that a permutation $\sigma: \N \to \N$ of the integers \dff{fixes} an element of $n \in \N$ if $\sigma(n) = n$. Let $\fix$ be the subgroup of all permutations of $\N$ that fix all but a finite number of elements of $\N$.  Let $(f_n)_{n \in \N}$ be a collection of densities and consider the product space   $\big([0,1]^{\N}, \borel, \bigotimes_{n\in \N}f_n\big)$.    The group $\fix$ acts on  this space via $\sigma(x)_n = x_{\sigma(n)}$ for all $x \in [0,1]^{\N}$ and all $n \in \N$.

  \begin{lemma}
  	\label{exch}
  	Let $(f_n)_{n \in \N}$ be a sequence of densities all with domain $[0,1]$.   Set 
  	$$g_n(x) :=  \essinf_{x\in[0,1]}{f_n(x)} \text{ and }  G_n(x):=   \esssup_{x\in[0,1]}{f_n(x)},$$
  	where these essential bounds are taken with respect to Lebesgue measure.    
  	If for all $n \in \N$, we have the following strict bounds
  	\[
  	0<\inf_{n\in\N} g_n(x) \leq \sup_{n\in\N} G_n(x) <\infty,
  	\] 
  	then $\left(\Omega,\mathcal{B},\bigotimes_{n\in\N}f_n, \fix \right)$ is ergodic. 
  \end{lemma}

  In our proof of Lemma \ref{exch} we will verify a condition that Aldous and Pitman \cite[Condition (c)]{aldo-pit} refer to as \emph{tameness}.  Let $u \wedge v$ denote the minimum of two real numbers $u,v \in \R$.

  \begin{theorem}[Aldous and Pitman]
  	\label{AP}
  	Let $(A, \mathcal{A})$ be a Borel measurable space endowed with a sequence of probability measures $(\mu_n)_{n\in \N}$.  Set 
  	$$	\mathcal{A}_0:=  \bigcap_{n \in \N}\left\{C \in \mathcal{A}:  \mu_n(C)\in \{0,1\}\right\}.$$
  	Suppose there exists a probability measure $\nu$ on $(A, \mathcal{A})$ such that for every $B \notin \mathcal{A}_0$ there exists  $\delta >0$ such that
  	$$\sum_{n \in \N} r_n(B, \delta)  = \infty,$$
  	where  
  	\begin{equation}
  		\label{RN} 	 	
  		r_n(B,\delta):=\inf\left\{\mu_n(B\setminus C) \wedge \mu_n\left(B^c\setminus C\right):\ \nu(C) \leq \delta \right\}.
  	\end{equation}
  	Then the $\fix$-action on the product space $(A^{\N}, \bigotimes_{n \in \N} \mathcal{A}, \bigotimes_{n \in \N} \mu_n)$ is ergodic. 
  \end{theorem}
  \begin{remark}
  	\label{Rtame}
  	Theorem \ref{AP} follows from \cite[Theorem 1.12]{aldo-pit}.   
  	A sequence of probability measures satisfying the conditions of Theorem \ref{AP} is said to be \dff{$\nu$-tame}. \erk
  \end{remark}
  Recall that $\les$ denotes the usual Lebesgue measure on $\R$.
  
  \begin{proof}[Proof of Lemma \ref{exch}]
  	Let $c \in (0,1)$ be such that for Lebesgue-almost all $u\in [0,1]$ for all $n\in\mathbb{Z}$, we have $c<f_n(u)<c^{-1}$.  For all Borel sets $B \subset [0,1]$, set  
  	$$ \mu_n(B) =  \int_B f_n(u) du = \int_B f_n(u) d\les(u),$$
  	so that 	
  	\[
  	c\les(B)\leq \mu_n(B)\leq c^{-1}\les(B).
  	\] 
  	
  	It is easy to check that for all $B\in\mathcal{B}$ with $0<\les(B)<1$ if $$\delta(B)= \delta:= (\les(B) \wedge\les(B^c)) /2,$$ then for all $n\in\N$, we have $r_n(B,\delta) \geq c \delta,$ where $r_n$ is defined as in \eqref{RN}.   
  	Thus tameness is verified with Lebesgue measure and by Theorem \ref{AP} and Remark \ref{Rtame}, we obtain the desired ergodicity. 
  \end{proof}

 \begin{theorem} 
 	\label{perm-type}
 	Let $\lambda  \in (0,1)$ and  $(f_n)_{n\in \N}$ be the sequence of densities defined  in Section \ref{def}. Then the system  $\left([0,1]^{\N},\mathcal{B}_{+},\bigotimes_{n\in\N}f_n,\fix\right)$ is of type-$\mathrm{III}_\lambda$. 
 \end{theorem}
 
 \begin{remark}
 	\label{form-perm}
 	It is straightforward to verify that for every $\sigma \in \fix$, if $\mu=\bigotimes_{n\in\N}f_n$ then for $\mu$-almost every $x \in [0,1]^{\N}$, we have 
 	\begin{equation*}
 		\frac{d\mu\circ \sigma}{d\mu}(x)=\prod_{n\in\N}\frac{f_{\sigma(n)}\left(x_n\right)}{f_n\left(x_n\right)},
 	\end{equation*}
 	where  this product has only a finite number of non-unit terms, since $\sigma$ fixes all but a finite number of integers.   For example, if $\sigma_{a,b}$ is the transposition of $a,b\in\mathbb{Z}$, then
 	\begin{equation}
 		\label{transposition}
 		\frac{d\mu\circ \sigma_{a,b}}{d\mu}(x)=\frac{f_a\left(x_b\right)f_b\left(x_a\right)}{f_a\left(x_a\right)f_b\left(x_b\right)}.
 	\end{equation} \erk
 \end{remark}

 We will use Lemma \ref{lem: CHP} to prove Theorem \ref{perm-type} by defining partial transformations that will be given by a finite number of disjoint transpositions so that an explicit computation can be performed via \eqref{transposition}. 
 Let $N >0$ be an integer.  For  a finite number of intervals $I_{0},...,I_{N} \subseteq [0,1]$, we say that the subset of $\Omega_{+}= [0,1]^{\N}$ given by  
 \[
 \ns{x\in \Omega_{+}: x_k\in I_k \text{ for every integer }   k \in [0,N] }
 \]
 is a \dff{cylinder set specified up to $N$};  if all the intervals $I_{0}, \ldots, I_N$ have rational endpoints, we say that it is \dff{rational}. The collection $\mathcal{G}$ of all rational cylinders is a countable semi-ring, and  the ring generated by $\mathcal{G}$ is dense in $\mathcal{B}_{+}$ the product sigma-algebra for $\Omega_{+}$.   
 
 \begin{proof}[Proof of Theorem \ref{perm-type}]
 	
 	We already have ergodicity from Theorem \ref{exch} and nonsingularity follows from Remark \ref{form-perm}.  
 	
 	Since for all $\sigma\in  \fix$, we have that 
 	$$\log\left(\frac{d\mu\circ \sigma}{d\mu}\right)\in (\log\lambda)\mathbb{Z} = \ns{ n\log(\lambda): n \in \Z},$$  
 	it is suffices to show that $\log\lambda$ is an essential value; we will  accomplish this  by verifying the conditions of Lemma \ref{lem: CHP} with $\mathcal{G}$ the rational cylinders and $\delta=1/2$.

 	Let   $\mathbf{I} \in\mathcal{G}$ be a cylinder set specified up to $N$.   Recall the sets $A_n, B_n \subset [0,1]$ that were used to define $f_n$ in Section \ref{def}.   Set  $$C_n:=[0,1]\setminus\left(A_n\cup B_n\right).$$  In what follows it will be convenient to use the language of random variables.  We endow measurable  space $(\Omega_{+}, \borel_{+})$  with the probability $\P = \mu$ and the expectation
 	$$ \E Y = \int Y(x) d\P(x)$$
 	for a random variable $Y: \Omega \to \R$.
 	Thus with $X: \Omega \to \Omega$ as  the identity $X(x) = x$ the sequence $X=(X_n)_{n \in \Z}$ is a sequence of continuous independent random variables with density functions given by $(f_n)_{n \in \N}$.    
 	
 	For $n \geq N+1$ and $n \not \in \ns{2^k: k \in \N}$, let  $Y_n:\Omega\to \{-1,0,1\}$ be defined by
 	\begin{align*}
 		Y_n:&=\ind_{C_n}\left(X_n\right)\ind_{A_n}\left(X_{2^n}\right)-\ind_{A_n}\left(X_n\right)\ind_{C_n}\left(X_{2^n}\right)\\
 		&=\ind[(X_n,X_{2^n}) \in C_n\times A_n]-\ind[(X_n,X_{2^n})\in A_n\times C_n].
 	\end{align*}
 	We also set $Y_{2^k} \equiv 0$ for all $k \in \N$.
 	We claim that 
 	\begin{equation}
 		\label{claim:infinite}
 		\lim_{M\to\infty}\sum_{k=N+1}^M Y_k=\infty \quad  \text{in probability.}
 	\end{equation} 
 	
 	Since $C_n\subset C_{2^n}$ and $A_{2^n}\subset A_n$, we have $$\int_{C_n}f_{2^n}(u)du=\left|C_n\right|$$ and 
 	\[
 	\int_{A_n}f_{2^n}(u)du=\int_{A_n\setminus A_{2^n}}1du+\lambda\left|A_{2^n}\right|.
 	\]
 	Since $X$ is an independent sequence, we have
 
 \begin{align*}
 	\mathbb{E}Y_n&=\int_{C_n} f_n(u)du\cdot \int_{A_n}f_{2^n}(u)du-\int_{A_n} f_n(u)du\cdot \int_{C_n}f_{2^n}(u)du\\
 	&=(1-\lambda)\left|C_n\right|\left(\left|A_n\right|-\left|A_{2^n}\right|\right)\\
 	&=\frac{1-\lambda}{(n+4)\log (n+4)}\left(1-\frac{\lambda+1}{(n+4)\log (n+4)}\right)+O\left(\frac{1}{2^nn}\right)
 \end{align*}
 for $n \not \in \ns{2^k: k \in \N}$.    Recall that  $Y_{2^k} \equiv 0$ for all $k \in \N$.
 	Thus $\sum_{k=N+1}^M\mathbb{E}Y_k\sim (1-\lambda)\log\log M$ as $M\to\infty$,  and since $0<\lambda<1$, we have $\lim_{M\to\infty}\sum_{k=N}^M\mathbb{E}Y_k=\infty$.  
 	
 	   Since for each $n \not \in \ns{2^k: k \in \N}$ the random variable $Y_n$ is a function of $(X_{n}, X_{2^n})$ the random variables $(Y_n)_{n \in \N}$ are independent,    a similar calculation gives that
 	\begin{align*}
 		\var\Big(\sum_{k=N+1} ^M Y_k\Big) &= \sum_{k=N+1} ^m \var(Y_k) \\
 		&=\sum_{k=N+1}^M\left(\mathbb{E}\left(Y_k^2\right)-\mathbb{E}\left(Y_k\right)^2\right)\\\
 		&= (\lambda+1+o(1))\log\log M. 
 	\end{align*} 
 	In particular, $$\var\left(\sum_{k=N+1}^MY_k\right)=o\left(\left(\mathbb{E}\left(\sum_{k=N+1}^MY_k\right)\right)^2\right) \text{ as } M\to\infty.$$  Hence Chebyshev's inequality gives, 
 	\begin{align*}
 		\P\left(\sum_{k=N+1}^MY_k<\sqrt{\log\log M} \right)
 		&= O\left(\frac{\var\left(\sum_{k=N+1}^MY_k\right)}{\left(\mathbb{E}\left(\sum_{k=N+1}^MY_k\right)\right)^2}\right) \ \text{as}\ M\to\infty. 
 	\end{align*}
 	Hence \eqref{claim:infinite} holds.

 	The divergence to infinity in \eqref{claim:infinite} and the fact that the random variables in the sum are $\{-1,0,1\}$ valued implies 
 	there  exists 
 	$M>N+1$ such that the set 
 	\[
 	\mathbf{E}:=\Big\{x\in\Omega_{+}: \text{there is an integer  $K \in [N+1,M]$ with } \sum_{j=N+1}^KY_j=1 \Big\},
 	\]
 	satisfies $\P(\mathbf{E})>\frac{1}{2}$.   Note that $\mathbf{E}$ depends on the random variables $(X_{N+1}, X_{2^{N+1}}), \ldots, (X_{M},X_{2^M})$.  
 	For $x\in\textbf{E}$ set 
 	\[\tau(x):=\min\Big\{n\geq N+1: \sum_{j=N+1}^n Y_j(x)=1\Big\}.
 	\]
 	
 	Now we  specify  a partial transformation $V$ with domain $\mathbf{D}:=\mathbf{E}\cap \mathbf{I}$ which satisfies the conditions of Lemma \ref{lem: CHP}.    	
 	Let $V:\mathbf{D}\to \mathbf{I}$ be defined by 
 	$$
 	(Vx)_j:=\begin{cases}
 	x_{2^j}, & j\in [N+1,\tau(x)] \text{ and }   Y_j(x)\neq 0,\\
 x_{\log_2 j}, & \log_2 j \in [N+1,\tau(x)]\cap \mathbb{N} \text{ and }   Y_{\log_2 j}(x)\neq 0,\\
 	x_j, & \text{otherwise}.
 	\end{cases}
 	$$
 	Thus $V$ is simply the application of the transposition  of 
 	certain dyadic pairs;  
 	$Vx$ is the sequence $x$, except that  $x_j$ is swapped with $x_{2^j}$ under certain conditions on $j$.

 	Since $\mathbf{I}$ is specified up to $N$,  the events  $\mathbf{E}$ and $\mathbf{I}$ depend on a disjoint collections of the coordinates $X =(X_n)_{n \in \N}$ and are thus independent. Therefore we have
 	\[
 	\P(\mathbf{D})=\P(\mathbf{E})\P(\mathbf{I})>\frac{1}{2}\P(\mathbf{I}),
 	\]
 	verifying the first condition of Lemma \ref{lem: CHP}. Secondly,  for $x\in\mathbf{D}$, we have
 	\begin{align*}
 		\frac{d\mu\circ V}{d\mu}(x)&=\prod_{k\in [N+1,\tau(x)],\ Y_k(x)\neq 0}\frac{d\mu\circ \sigma_{k, 2^k}}{d\mu}(x)\\
 		&=\prod_{k\in [N+1,\tau(x)],\ Y_k(x)\neq 0} \frac{f_k\left(x_{2^k}\right)}{f_k\left(x_k\right)}\frac{f_{2^k}\left(x_k\right)}{f_{2^k}\left(x_{2^k}\right)}\\
 		&=\prod_{k\in [N+1,\tau(x)],\ Y_k(x)\neq 0}\lambda^{Y_k(x)},
 	\end{align*}
 	where the 
 	last 
 	equality 
 	follows from the fact that if $Y_k(x)=1$, then $x_{2^k}\in A_k$ and $x_{k}\notin (A_k\cup B_k)$ so that  $f_k\left(x_{2^k}\right)=\lambda$ and $f_k\left(x_k\right)=1$ and similarly if $Y_k(x)=-1$, then $f_k(x_{2^k}) =1$ and $f_k(x_k)=\lambda$.  Hence 
 	\[
 	\frac{d\mu\circ V}{d\mu}(x)=\lambda^{\sum_{k=N+1}^{\tau(x)}Y_k(x)}=\lambda,
 	\]
 	and we have verified (an epsilon free version of) the second condition of the lemma.

 	Clearly $V$ is nonsingular and by construction $(x, Vx) \in \mathcal{O}_{\fix}$ for all $x \in \mathbf{D}$.    It remains to verify that $V$ is injective.   Let $x,x'\in\mathbf{D}$ be such that $Vx=Vx'$ and assume without loss of generality that $\tau(x)\leq \tau(x')$.  Since for all $N+1\leq k\leq \tau(x)$, we have
 	\[
 	Y_k\left(Vx\right)=-Y_k(x)
 	\] 
 	it follows that  for all $N+1\leq k\leq\tau(x)$, we have $Y_k(x)=Y_k(x')$. Hence 
 	\[
 	\sum_{k=N}^{\tau(x)}Y_k(x')=\sum_{k=N}^{\tau(x)}Y_k(x)=1
 	\]
 	so that the minimality of $\tau$ gives that  $\tau(x)=\tau(x')$. In addition, the subsets where $V$ fixes the coordinates of $x$ and $x'$ are the same, since 
 	$$\left\{k\in [N+1,\tau(x)]:\  Y_k(x)\neq 0\right\} = \left\{k\in [N+1,\tau(x)]:\  Y_k(x')\neq 0\right\}.$$ 
 For the other indices, by definition, $Vx$ and $Vx'$ are permutations of the coordinates of $x$ and $x'$, respectively.   Since $Vx = Vx'$ and  permutation are injective, it follows that $x=x'$ as desired. 
 \end{proof}

Let $\lambda \in (0,1)$.   Let $(\Omega_{+}, \borel_+, \m_{+}, S)$ be the one-sided Bernoulli shift from Section \ref{def}.
The \dff{tail} equivalence relation on $\Omega_+$ is given by
$$ \mathcal{T} := \big\{ (x,x' ) \in \Omega_+ \times \Omega_+:  \text{there exists $n \in \N$ with $S^nx = S^n x'$} \big\}.$$
The \dff{exchangeable} equivalence relation is given by
 $$\mathcal{E}:=\left\{(x,x')\in \Omega_+\times\Omega_+: \text{there exists } \sigma\in \fix \text{ with } \sigma (x)=x' \right\}.$$
\begin{remark}
\label{coarse}
 Note that $\mathcal{E}$ is countable and  coarser than $\mathcal{T}$. \erk
\end{remark}
  Let  $S_{\varphi_\mu}$ be its discrete Maharam extension.   Let $\psi: \Omega_+ \times \Omega_+ \to \Z$ be the \dff{tail cocycle} associated to $\varphi_\mu$, defined by
 \[
 \psi(x,x'):=\sum_{n=0}^\infty\left(\varphi_\mu\circ S^n(x)-\varphi_\mu\circ S^n(x')\right).
 \]
\begin{remark}
	\label{tail-cond}
 Note that for all $(x,z),(x',z')\in \Omega_+\times\bb{Z}$, there exists $n\in\mathbb{N}$ such that $S_{\varphi_\mu}^n(x,z)=S_{\varphi_\mu}^n(x',z')$ if and only if $(x,x')\in \mathcal{T}$ and 
 \[
 z+\psi(x,x')=z'. 
 \]
 \end{remark}
 For $\sigma\in \fix$, let $\psi_{\sigma}(x):=\psi(x,\sigma(x))$.   
The  \dff{discrete Maharam extension of the $\fix$-action}  on $\Omega_+\times \bb{Z}$ is given by
 \[
 \sigma(x,n):=\Big(\sigma(x),n-\log_\lambda\frac{d\mu\circ \sigma}{d\mu}\Big).
 \] 
 \begin{lemma}
 	\label{center}
 	    Let $\lambda \in (0,1)$.   Let $(\Omega_{+}, \borel_+, \m_{+}, S)$ be the one-sided Bernoulli shift from Section \ref{def}.
 	For every $\sigma\in\fix$, we have  $$\psi_\sigma=\log_\lambda\frac{d\mu\circ \sigma}{d\mu},$$
 	where $\psi$ is the associated tail cocycle.
 	  Furthermore,  if the Maharam extension of the $\fix$-action 
 	is ergodic with respect to $\widetilde{\mu_+}$, then 
 	$$\bigcap_{n=1}^\infty S_{\varphi_\mu}^{-n}\left(\mathcal{B}([0,1] ^{\N})\otimes\B(\Z)\right)=\{\emptyset,\Omega_+\times \bb{Z}\} \bmod \widetilde{\mu_+},$$
 	so that discrete Maharam extension of the left-shift is a $K$-auto\-morphism. 
 \end{lemma}
 
 \begin{proof}[Proof of Theorem \ref{duality-checked}]
 	Theorem \ref{perm-type} together with  Theorem \ref{duality} implies that the Maharam extension of the $\fix$-action is ergodic, thus the hypothesis of Lemma \ref{center} is satisfied and the discrete Maharam extension of the left-shift is a $K$-automorphism and by Lemma \ref{ST-K} it is ergodic; finally, by Theorem \ref{duality} this implies that the Bernoulli shift is of type-$\mathrm{III}_{\lambda}$.
 	\end{proof}

 \begin{proof}[Proof of Lemma \ref{center}]

 	Let $\sigma \in \fix$ and suppose that for all $m \geq n$, the integer $m$ is fixed by $\sigma$.    Then for all $m\geq n$ and $x\in\Omega_+$, we have $S^m\sigma(x)=S^m(x)$.  Thus for $\mu$-almost all $x \in \Omega_{+}$, we have

 	\begin{align*}
 		\psi(x,\sigma(x))&=\sum_{k=0}^{n-1}\left(\varphi_\mu\circ S^k(x)-\varphi_\mu\circ S^k(\sigma(x))\right)\\
 		&= \log\left(\prod_{k=0}^{n-1}\frac{d\mu\circ T}{d\mu}\left(T^kx\right)\right)-\log\left(\prod_{k=0}^{n-1}\frac{d\mu\circ T}{d\mu}\left(T^k\sigma(x)\right)\right)\\
 		&=\log\left(\left(T^n\right)'(x)\right)-\log\left(\left(T^n\right)'(\sigma(x))\right).
 	\end{align*}
 	Again, since $\sigma$ fixes all integers $m \geq n$ and  $f_{k-n}\equiv 1$ for all $k\leq n$, for $\mu$-almost all  $x\in\Omega_+$, we have
 	\begin{align*}
 		\frac{\left(T^n\right)'(x)}{\left(T^n\right)'(\sigma(x))}&=\prod_{k=1}^n\frac{f_{k-n}\left(x_k\right)}{f_{k}\left(x_k\right)}\frac{f_{k}\left(x_{\sigma(k)}\right)}{f_{k-n}\left(x_{\sigma(k)}\right)}\\
 		&=\prod_{k=1}^n\frac{f_{k}\left(x_{\sigma(k)}\right)}{f_{k}\left(x_k\right)}\\
 		&=\frac{d\mu\circ \sigma}{d\mu}(x).
 	\end{align*}
 The first claim is verified.  
 	
 The first claim gives that the equivalence relation $\mathcal{R}$ on $\Omega_+\times \bb{Z}$ given by
 	\[
 	 \big\{((x,z),(x',z')) :\ \exists \sigma\in\fix \text{  with }  \sigma(x)=x' \text{and}\ z'=z-\psi(x,\sigma(x)) \big\}
 	\]
 is in fact the tail equivalence relation for the Maharam extension of the $\fix$-action	given by
 $$  \big\{((x,z),(x',z')) :\ \exists \sigma \in \fix  \text{ such that } \sigma(x,z) = (x',z')\big\}.$$
 By Remarks \ref{coarse} and \ref{tail-cond}, $\mathcal{R}$ is countable and is coarser than 
 the  
 tail relation of $S_{\varphi_\mu}$, given by
 $$  \big\{((x,z),(x',z')) :\ \exists n \in \N  \text{ such that } S_{\varphi_\mu}^n(x,z) = (x',z')\big\}.$$
  Therefore, $\mathcal{I}$, the sigma-field of invariant sets for the equivalence relation $\mathcal{R}$ contains the tail sigma-field for $S_{\varphi_\mu}$.   Hence  the assumption that the  Maharam extension of the $\fix$-action is ergodic yields the desired triviality. 
 \end{proof}

 \begin{remark}
It was proved in \cite[Theorem 3.2]{KosKMaharam} and \cite{DaniLem} that   a conservative  nonsingular half stationary Bernoulli shift on two symbols without an absolutely continuous invariant probability measure has a  Maharam
 	extension that is a $K$-transformation.  
 	The idea of Kosloff's proof and extensions given in \cite{DaniLem} do not extend in an obvious way in the setting of Theorem 	\ref{duality-checked}, since the tail equivalence relation for $S$ is not countable and not nonsingular; see \cite[Section 2.3] {KosKMaharam} for more details.   \erk
 \end{remark}

 \section{Countable amenable groups}
 \label{a-group}
 
 \subsection{Introduction}
 Benjamin Weiss and Andrey Alpeev  asked whether for every countable amenable group and $0\leq\lambda\leq 1$ there exists an explicit action of $G$ which is type-$\mathrm{III}_\lambda$.   We give a positive (partial) answer to this question in Theorem \ref{group-exists} by producing type-$\mathrm{III}_\lambda$ Bernoulli action of $G$ for $\lambda \in (0,1)$.    The  case where $\lambda =1$ is known from \cite{BjoKosVaes,VaesWahl}.

 Let $G$ be a countable group.    We  will sometimes  write $G\curvearrowright\left(\Omega,\F,\mu\right)$ to denote a group action on the measured space $(\Omega, \F, \mu)$.      The group $G$ acts on $\Omega =[0,1]^G$ via $$(T_g x)_h = (g x)_h := x_{g^{-1}h}$$ for all $x \in \Omega$ and all $g, h \in G$.    We endow $\Omega$ with the usual Borel product sigma-algebra, $\borel$.  Let $(f_i)_{i\in G}$ be a collection of probability densities with domain $[0,1]$ and let $\mu = \bigotimes_{i \in G} f_i$.  Then we say that $G$ is a \dff{Bernoulli action} on $(\Omega, \borel, \mu)$.

 Recall that an action $G\curvearrowright\left(\Omega,\F,\mu\right)$ is of \dff{stable type-$\mathrm{III}_\lambda$} if for every ergodic probability preserving $G$ action on $\left(\Omega',\mathcal{C},\nu\right)$, the diagonal action $G\curvearrowright \left(\Omega\times \Omega',\F\otimes \mathcal{C},\mu\otimes\nu\right)$ is ergodic and of  type-$\mathrm{III}_\lambda$.

For a finite set $F$, let $\num{F}$ denote its cardinality.   Recall that a \dff{F{\o}lner sequence} for a countable group $G$ is a sequence $(F_n)_{n \in \N}$ of finite subsets of $G$ such that  for all $g \in G$, we have
$$ \lim_{n \to \infty} \frac{ \num{ F_n \triangle g F_n} } { \num{F_n} } =0.$$ 
 F{\o}lner \cite{MR79220} proved that 
  $G$ is amenable if and only it admits it
  a
    F{\o}lner sequence. 
Recall the statement of  Theorem \ref{group-exists}:  
 \begin{theorem-nonum}
  	Let $G$ be a countable amenable group and $\lambda\in(0,1)$. There exists a product measure $\bigotimes_{g\in G}f_g$ on $[0,1]^G$ such that the corresponding Bernoulli action is nonsingular, ergodic and of stable type-$\mathrm{III}_\lambda$. 
\end{theorem-nonum}
We will prove Theorem \ref{group-exists}  by adapting the densities defined in Section \ref{def} to the amenable group setting.  The main difference in this more general setting is that when we no longer have a notion of the $K$-property.

 \subsection{The explicit construction}
 \label{def-group}

 We will now proceed with the construction of the density functions on $[0,1]$. Let $\lambda\in (0,1)$ and $G=\left\{g_n\right\}_{n\in\bb{N}}$ be a countable amenable group which we enumerate. A straightforward application of F{\o}lner's characterization of amenability implies that  there exists a pairwise disjoint   F{\o}lner sequence $(F_n)_{n \in \N}$  satisfying the property that for all  integers $0 \leq k <n$, we have
 \begin{equation}\label{eq: Folner}
 	\max\left(\frac{ \num{F_n\triangle g_k^{-1}F_n}} {\num{ F_n}},\frac{\num{F_n\triangle g_kF_n}}{\num{F_n}} \right)<\frac{1}{n}, 
 \end{equation}
 and  that the union of F{\o}lner sets leaves an infinite subset of $G$, so that $G \setminus \bigcup_{n \in \N} F_n$ is infinite.   We will also assume that $\num{F_n} \geq 4$ for all $n \in \N$.

Recall that $\les$ denotes Lebesgue measure.  Let $(A_n)_{n \in \N}$ and  $(B_n)_{n \in \N}$ be a decreasing  sequences of open intervals of $[0,1]$ such that $A_1$ and $B_1$ are disjoint, that is,  any two of $A_m$ and $B_n$ are disjoint; furthermore, we specify that 
  $$\les\left(A_n\right)=\lambda^{-1}\les\left(B_n\right)=\frac{1}{n\log (n+1)}\frac{1}{ \num{F_n}}.$$   Since $\num{F_n}\geq 4$ it follows that $$\les\left(B_n\right)+\les\left(A_n\right)<1.$$

 Let $(\hat{f}_n)_{n \in \N}$ be a sequence of $\{\lambda^{-1},1,\lambda\}$-valued densities on $[0,1]$ given   by
 $$
 \hat{f}_n(u):=\begin{cases}
 \lambda, &u\in A_n,\\
 \lambda^{-1},\ &u\in B_n,\\
 1, &  u\in [0,1]\setminus \left(A_n\cup B_n\right),
 \end{cases}
 $$
 and $(f_g)_{g\in G}$  be the sequence of functions  indexed by $G$ given by
 \begin{equation}
 	\label{def of f_g}
 	f_g:=\begin{cases}
 		\hat{f}_n, & \text{if there exists $n \in \N$ with} \ g\in F_n,\\
 		1_{[0,1]}, & \text{otherwise}.
 	\end{cases}
 \end{equation}
 
 We will show that the Bernoulli action $G\curvearrowright\left(\Omega,\B,\mu\right)$ witnesses Theorem \ref{group-exists}, where $\Omega = [0,1]^G$ and $\mu = \bigotimes_{g \in G} f_g$.
 
 \subsection{The proof of Theorem \ref{group-exists} }
 
Our approach to the proof of Theorem \ref{group-exists} is similar to the proof of Theorem \ref{main-exists} and in what follows we display the main components, and highlight the differences with regards to our treatment of Theorem \ref{main-exists}, and  postpone the proofs of the components parts until Section \ref{remain}.

  \begin{lemma}
  	\label{group-con}
  Let $\lambda \in (0,1)$.  	The Bernoulli action defined in Section \ref{def-group} is nonsingular and conservative.
  \end{lemma}
 
 Lemma \ref{group-con} follows from similar calculations given earlier for the proofs of Lemma \ref{nonsing} and \ref{cons}.

 Again, in order to prove that $\lambda$ is an essential value, we will  analyze the action  the permutation group rather than the shift action of the group itself.    Denote by $\fix_G$ the group of finite permutations on $G$; that is, those that fix all but a finite number of elements of $G$.   This group acts on $[0,1]^G$ by setting  $(\sigma x)_g:=x_{\sigma(g)}$ for all $g \in G$ and all $x \in [0,1]^G$.  Note that  action is nonsingular with Radon-Nykodym derivative 
 \[
\sigma'(x)= \frac{d\mu\circ \sigma}{d\mu}(x)=\prod_{h\in G} \frac{f_h\left(x_{\sigma( h)}\right)}{f_h\left(x_h\right)}.
 \]
 Furthermore, we have the following version of Theorem \ref{perm-type}.
 \begin{proposition}\label{prop: Krieger S_G}  
 	Let $\lambda \in (0,1)$ and  $(f_g)_{g\in G}$ be the sequence of functions defined by \eqref{def of f_g} in Section \ref{def-group}.   Then $\big(\Omega,\mathcal{B},\bigotimes_{g\in G}f_g, \fix_G\big)$ is of Krieger type-$\mathrm{III}_\lambda$. 
 \end{proposition}

 In order exchange Proposition \ref{prop: Krieger S_G} for a statement about the Bernoulli group action we will apply the Hopf method argument as in Avraham-Re'em \cite[Section 4]{Nachi}.  Let $G$ be a countable group.   We write  $g_n\to\infty$ if for every finite $H \subset G$ and for all $n$ sufficiently large we have  $g_n\notin H$.  Let $\Phi: G \to \R$.   If for all sequences $g_n \to \infty$, the  $ \lim_{n \to \infty}   \Phi(g_n)$ exists and has the same value $L$, then we write $\lim_{g \to \infty} \Phi(g) =L$.

    Let $(\Omega, \borel, \mu)$ be a standard measure space, endowed with a complete separable metric $d$ on $\Omega$ that generates $\borel$.  Suppose $G\curvearrowright\left(\Omega,\B,\mu\right)$.   A pair of points $(x,x')\in \Omega\times \Omega$ is \dff{an asymptotic pair} if 
 \[
 \lim_{g\to\infty}d\left(gx,gx'\right)=0.
 \]
 We say that an action of another countable group $\Gamma$  on $\left(\Omega,\mathcal{B},\mu\right)$ is \dff{an action by $G\curvearrowright\left(\Omega,\borel,\mu\right)$ asymptotic pairs} if for every $\gamma\in \Gamma$ for $\mu$-almost every $x\in \Omega$, we have
 \[
 \lim_{g\to\infty}d\left(gx,g\gamma x\right)=0.
 \]
 \begin{remark} 
 	\label{key-example}
 	 Consider the Bernoulli action  $G\curvearrowright\left(\Omega,\B,\mu\right)$ given in Section \ref{def-group}.  Note that 
 	$\fix_G\curvearrowright\left(\Omega,\B,\mu\right)$ is an action by $G\curvearrowright\left(\Omega,\B,\mu\right)$ asymptotic pairs.  
 	Let $\sigma\in\fix_G$ and write $$H:=\left\{h\in G:\ \sigma(h)\neq h\right\}.$$ For every $x\in \Omega$, for all $g\in G$, the set $\left\{h\in G: (gx)_h\neq (g\sigma x)_h\right\}$ is contained in $gH$, which is finite, since $H$ is finite.   Hence for all $x\in \Omega$ and any $g_n \to \infty$, we have that for all $n$ sufficiently large 
 	\[
 	d\left(g_nx,g_n\sigma x\right) = 0.
 	\]
 	\erk
 \end{remark}
Let $\lambda \in (0,1)$.  Consider the action $G\curvearrowright\left(\Omega,\F,\mu\right)$ as in Section \ref{def-group}.   Recall that we defined the discrete Maharam extension of a group action in Section \ref{ME-Z}.    Thus the \dff{discrete Maharam extension of the $G$-action}  on $\Omega\times \Z$ is given by
 \[
 g(x,n):=\big(g(x),n-\log_\lambda\left(g'(x)\right)\big)
 \] 
 and
 the \dff{discrete Maharam extension of the $\fix_G$-action}  on $\Omega \times \Z$ is given by
 \[
 \sigma(x,n):=\big(\sigma(x),n-\log_\lambda\left(\sigma'(x)\right)\big).
 \] 
 
 \begin{lemma}
 	\label{asy-pairs}
 	Let $\lambda \in (0,1)$.  Consider the action $G\curvearrowright\left(\Omega,\borel,\mu\right)$ as in Section \ref{def-group}.   The discrete Maharam extension of the  $\fix_G$-action on $\Omega \times \Z$  is an action of   asymptotic pairs of the discrete Maharam  extension of the $G$-action on $\Omega \times \Z$.  
 \end{lemma}
 
 \begin{proof}[Proof of Theorem \ref{group-exists}]
 	Let $\lambda \in (0,1)$.  Consider the action $G\curvearrowright\left(\Omega,\borel,\mu\right)$ as in Section \ref{def-group}.
 	Let $G\curvearrowright\left(\Omega',\mathcal{C},\nu\right)$ be an additional 
 	ergodic probability preserving action. 
 	Since 
 	\[
 	\frac{d(\mu\otimes \nu)\circ g}{d(\mu\otimes\nu)}(x,y)=\frac{d\mu\circ g}{d\mu}(x)=g'(x),
 	\]
 	the discrete Maharam  extension of the diagonal action on the product space  $G\curvearrowright\left( \Omega \times \Omega',\borel \otimes\mathcal{C},\mu\otimes\nu\right)$ is given by
 	\[
 	g(x,y,n)=\big(gx,gy,n-\log_\lambda(g'(x))\big),
 	\]
 	 for all $(x,y,n)\in \Omega \times \Omega' \times \Z$.
 	We specify that for $\sigma \in \fix_G$ the action on the product space $\Omega \times \Omega'$ is given by ignoring the second coordinate so that
 	$$ \sigma(x,y) = (\sigma x, y).$$
 	Thus the discrete Maharam extension of $\fix_G$ is given by
 	\[
 	\sigma(x,y,n)=\big(\sigma x,y,n-\log_\lambda(\sigma'(x))\big).
 	\]
 As in Remark \ref{key-example}, the $\fix_G$-action is an action by asymptotic pairs of the discrete  Maharam extension of the diagonal $G$-action.

 Let $F:\Omega\times \Omega'\times\bb{Z}\to [0,1]$ be a $G$-invariant function with respect to $\widetilde{\mu \times \nu}$, the measure associated with discrete Maharam extension.    We will show that $F$ is a constant. By \cite[Lemma 4.4]{Nachi}, 
 $F$ is also $\fix_G$-invariant.   Proposition \ref{prop: Krieger S_G} together with  Theorem \ref{duality} give that  discrete Maharam extension of $\fix_G\curvearrowright\left(\Omega,\B,\mu\right)$  is ergodic.  Hence it follows from the definition of $\fix_G$-action on $\Omega\times \Omega'\times \Z$, which is the identity on $\Omega'$, and ergodicity that there exists $\psi:\Omega'\to [0,1]$ such that $F(x,y,n)=\psi(y)$.  Since  $F$ is $G$-invariant 
 it follows that for almost every $y\in \Omega'$, that  $\psi(y)=\psi(gy)$  for all $g\in G$,
  from which the assumed ergodicity of $G\curvearrowright\left(\Omega',\mathcal{C},\nu\right)$ yields that $F$ is a constant almost surely. 
 
 Hence $$G\curvearrowright\left(\Omega \times \Omega'\times \Z,\borel\otimes\mathcal{C}\otimes\Z,\widetilde{\mu\otimes\nu}\right)$$ is ergodic. From another application of Theorem \ref{duality} we obtain   that $G\curvearrowright\left(\Omega \times \Omega',\borel\otimes\mathcal{C},\mu\otimes\nu\right)$ is of type-$\mathrm{III}_\lambda$.  Since the system $(\Omega', \mathcal{C}, \nu)$ was arbitrary, we have the desired stability.   
 \end{proof}	
 
 \subsection{The remaining proofs for Theorem \ref{group-exists}}
 \label{remain}

 \begin{proof}[Proof of Lemma \ref{group-con}]
 	Let $g\in G = \ns{g_k}_{k \in \N}$ and $n:=n(g)\in\mathbb{N}$ such that $g=g_n$. We will show that the action is nonsingular by applying Kakutani's theorem on the equivalence of product measures; furthermore we will obtain that there exists $K(\lambda)\geq \frac{6}{\lambda}$ such that 
 	\begin{equation}\label{eq: Amenable conser}
 		\sum_{h\in G}\int_0^1\left(\sqrt{f_{g^{-1}h}}(u)-\sqrt{f_h}(u)\right)^2du<K(\lambda)\log\log(n(g)+1).
 	\end{equation}

 Let 
 $$H_k(g):=F_k\triangle \left(g F_k\cup g^{-1}F_k\right).$$
 	Note that if  $h\notin \bigcup_{k\in\N} H_k(g)$, then $f_{g^{-1}h}=f_h$. Thus for $h \in H_k(g)$, we have
 	\begin{align*}
 		\int_0^1\left(\sqrt{f_{g^{-1}h}}-\sqrt{f_h}\right)^2du&\leq \int_0^1\left(\sqrt{\hat{f}_k}-1\right)^2du\\
 		&\leq \lambda^{-1} \left(\les\left(A_k\right)+\les\left(B_k\right)\right)\\
 		&\leq \frac{2}{\lambda k\log(k+1)}\frac{1}{ \num{F_k}}, 
 	\end{align*}
 	the second inequality follows from the simply inequalities $$\sqrt{\lambda^{-1}}-1,1-\sqrt{\lambda}<\lambda^{-1/2}.$$
 	 Consequently,
 	\begin{align*}
 		\sum_{h\in G}\int_0^1\left(\sqrt{f_{g^{-1}h}}-\sqrt{f_h}\right)^2du&=\sum_{k=1}^\infty \sum_{h\in H_k(g)}\int_0^1\left(\sqrt{f_{g^{-1}h}}-\sqrt{f_h}\right)^2du\\
 		&\leq 	2\lambda^{-1}\sum_{k=1}^\infty \sum_{h\in H_k(g)}\frac{\num{H_k(g)}}{\num{F_k}}\frac{1}{k\log(k)}.
 	\end{align*}
 	By \eqref{eq: Folner}, 
 	$$
 	\frac{\num{H_k(g)}}{\num{F_k}}\leq\begin{cases}
 	\frac{2}{k}, & k\geq n(g),\\
 	3, & 1\leq k\leq n(g).
 	\end{cases}
 	$$
 	Hence 
 	\begin{align*}
 		\sum_{h \in G}\int_0^1\left(\sqrt{f_{g^{-1}h}}-\sqrt{f_h}\right)^2du&\leq \frac{6}{\lambda}\sum_{k=1}^{n(g)}\frac{1}{k\log(k+1)}+\frac{4}{\lambda}\sum_{k=n(g)}^\infty\frac{1}{ k^2\log(k+1)}\\
 		&\leq \frac{6}{\lambda}\left(\log\log(n(g)+1)+n(g)^{-1}\right),
 	\end{align*}
 	which concludes the proof of \eqref{eq: Amenable conser} from which nonsingularity follows.

 	Note that if $h \notin H_k(g)$, then 
 	\[
 	\int_0^1 \left(\frac{f_h(u)}{f_{g^{-1}h}}\right)^2f_h(u)du=\int_0^1f_h(u)du=1.
 	\]
If $h\in H_k(g)$, then 	as in the proof of Lemma \ref{lem: VW}, there exists $c(\lambda)>0$ such that  
 	\[
 	\int_0^1\left(\frac{f_h(u)}{f_{g^{-1}h}}\right)^2f_h(u)du\leq \exp\left(c(\lambda)\les\left(A_k\right)\right). 
 	\]
 	Thus
 	\begin{align*}
 		\int_{\Omega}\left(\frac{1}{\left(T_g\right)'(x)}\right)^2d\mu(x)&=\prod_{k=1}^\infty\prod_{h\in H_k(g)}\int_0^1 \left(\frac{f_h(u)}{f_{g^{-1}h}}\right)^2f_h(u)du\\
 		&\leq \exp\left(c(\lambda)\sum_{k=1}^\infty \les\left(A_n\right)\num{H_k(g)}\right)\\
 		&\leq 	 \exp\left(K(\lambda)c(\lambda)\log\log(n(g))+1)\right),
 	\end{align*}
 	where the last bound is the same as in the proof of \eqref{eq: Amenable conser}. 
 	
 	Hence by Markov's inequality,  the sequence of sets 
 	\[
 	C_n:=\left\{x \in \Omega:\ \left(T_{g_n}\right)'(x)<n^{-1}\right\}
 	\]
satisfies
 	\[
 	\mu\left(C_n\right)=\mu\ns{x \in \Omega:  ( (T_{g_n})'(x))^{-2}>n^{2}}  \leq \frac{\left[\log(n+1)\right]^{K(\lambda)c(\lambda)}}{n^2}.
 	\]
 	The first Borel-Cantelli lemma implies that for $\mu$-almost all $x\in \Omega$ there exists $N(x)\in\mathbb{N}$ such that for all $k>N(x)$, we have $\left(T_{g_k}\right)'(x)\geq k^{-1}$. Again, the divergence of the harmonic series gives
 	\[
 	\sum_{g\in G}\left(T_{g}\right)'(x)=\sum_{n\in \bb{N}}\left(T_{g_n}\right)'(x)=\infty,
 	\]
for $\mu$-almost every $x \in \Omega$. 
	By a routine extension of Hopf's criteria  \cite[Proposition 1.3.1]{AaroBook} to the case of a general countable group action,   the Bernoulli action is conservative. 
 \end{proof}

 Our proof of Proposition {\ref{prop: Krieger S_G} will be an adaptation of 
 	the 
 	proof of Theorem \ref{perm-type} and in our proof we will point out the necessary modifications.     
Consider the action $G\curvearrowright\left(\Omega,\B,\mu\right)$ given in Section \ref{def-group}.
 Given a finite subset $H\subset G$ and  open subintervals $\left\{I_h\right\}_{h\in H}$ with rational endpoints, we say that the set 
 \[
 \mathbf{I}=\left\{x\in \Omega: \text{for all $h \in H$ we have} \ x_h\in I_h\right\}.
 \]
 is a \dff{rational cylinder set determined on $H$}. 
 As in  Section \ref{perm}, the collection $\mathcal{G}$ of rational cylinders is a countable semiring and the ring generated by $\mathcal{G}$ is dense in $\borel$.

 \begin{proof}[Proof of Proposition \ref{prop: Krieger S_G}]
 We first note that the ergodicity of the nonsingular system $\big(\Omega,\mathcal{B},\bigotimes_{g\in G}f_g,\fix_G\big)$ follows from  Lemma \ref{exch}, since in its proof we can replace $\N$ be $G$.

We will show that $\lambda$ is an essential value for the $\fix_G$-action and this is done by verifying the conditions of Lemma \ref{lem: VW} with $\mathcal{G}$ the rational cylinders and $\delta=1/2$. Let $H\subset G$ 
be 
finite and $\mathbf{I}$ be a rational cylinder determined on $H$.     

In order to mimic the proof of Theorem \ref{perm-type}, we will need to fix certain useful enumerations of subsets of $G$.  
 	Enumerate $\bigcup_{n \in \N} F_n = (\f_n)_{n \in \N}$ so that each element of $F_m$ has a greater index than each element of $F_n$, if $m >n$.  Specifically,  write $n_{-1}=0$ and for $k\in \N$, let $n_k=n_{k-1}+(\num{F_k})$. Then for all $k\in\bb{N}$ we choose an enumeration  $$\ns{\f_j}_{j=n_{(k-1)}}^{n_k-1}=F_k.$$
 
  We will write for $\mathbf{A}_j=A_n$, $\mathbf{B}_j=B_n$, and $\mathbf{C}_j=C_n$ if $\f_j\in F_n$, or equivalently $j\in \left[n_{k-1},n_k\right)$.
 	Let $(h_j)_{j=1}^\infty$ be an enumeration of the countably infinite set $G\setminus \left(\bigcup_{n \in \N} F_n\right)$ and $N\in\mathbb{N}$ such that 
 	$$H\subset \Big( \{\f_j\}_{j=0}^{n_N-1}  \ \cup \  \{h_j\}_{j=0}^{n_N-1} \Big):= H'.$$

 	With these enumerations we define the random variables as in the proof of Theorem \ref{perm-type}.   
 	Let $X: \Omega \to \Omega$ be the identity function, $X(x) = x$ so that $(X_g)_{g\in G}$ are a collection of independent continuous random variables with densities $(f_g)_{g \in G}$.  
 	For  $j\in\bb{Z}$, let  $Y_j:\Omega\to \{-1,0,1\}$ be given by 
 	\[
 	Y_j=\mathbf{1}_{\mathbf{A}_j}\left(X_{h_j}\right)\mathbf{1}_{\mathbf{C}_j}\left(X_{\f_j}\right)+\mathbf{1}_{\mathbf{A}_j}\left(X_{\f_j}\right)\mathbf{1}_{\mathbf{C}_j}\left(X_{h_j}\right).
 	\]
 	For all $n \geq n_N$, let $Z_n:=\sum_{j=n_N}^nY_j$. 
 	
 	As in the proof of Theorem \ref{perm-type}, by elementary expectation-variance calculations it is easy to verify that
 	\begin{equation}
 		\label{claimed}
 	\mu( Z_j \geq 1) \to 1  \text{ as }  j \to \infty.
 	\end{equation} 
 	By \eqref{claimed} let $J\geq n_N$ be such  that the set 
 	\[
 	\mathbf{E}:=\left\{x\in\Omega:\ \exists k\in \left[n_N,J\right],\ Z_k=1\right\}
 	\]
 	satisfies $\mu(\mathbf{E})\geq\frac{1}{2}$. Let $\mathbf{D}:=\mathbf{I} \cap \mathbf{E}$.   Note that  $\mathbf{I}$ is determined on $H \subset H'$ and depends on $(X_g)_{g \in H}$ and $\mathbf{E}$ depends on $(X_g)_{g \in G \setminus H'}$.  Since $X$ is 
 	a sequence independent of random variables,  
 	the events $\mathbf{E}$ and $\mathbf{I}$ are independent so that $\mu(\mathbf{D}) \geq \frac{1}{2} \mu(\mathbf{I})$.  Let  $\tau:\mathbf{D}\to \left[n_N,J\right]$ be given by  $\tau(x):=\min\left\{l\in \left[n_N,J\right]:\ Z_l=1\right\}$ and $V:\mathbf{D}\to \mathbf{I}$ be given by 
 	$$
 	(Vx)_g:=\begin{cases}
 	x_{\f_j}, &\exists j\in \left[n_N,\tau(x)\right],\ Y_j\neq 0\  \text{and}\ g=h_j,\\
 	x_{h_j}, &\exists j\in \left[n_N,\tau(x)\right],\ Y_j\neq 0\  \text{and}\ g=\f_j,\\
 	x_g, &\ \text{otherwise}.
 	\end{cases}
 	$$ 
 	As in the proof of Theorem \ref{perm-type} an easy calculation shows that for all $x\in \mathbf{D}$, we have 
 	\[
 	\frac{d\mu\circ V}{d\mu}=\lambda^{Z_{\tau(x)}}=\lambda.
 	\]
 	Again, from the proof of Theorem \ref{perm-type},  it is routine to verify that $V$ is injective.
 	
 	By Lemma \ref{lem: CHP}, $\lambda$ is an essential value for the $\fix_G$-action and since the Radon-Nikodym derivatives of the $\fix_G$-action are in $\lambda^\mathbb{Z}$ we conclude that $\left(\Omega,\mathcal{B},\bigotimes_{g\in G}f_g,\fix_G\right)$ is of Krieger type-$\mathrm{III}_\lambda$. 
 \end{proof}

 \begin{proof}[Proof of Lemma \ref{asy-pairs}]
 	For $\sigma \in \fix_G$, $g \in G$, and for almost every $x \in \Omega=[0,1]^G$,  we have 
 	\begin{align*}
 		g'\circ \sigma(x)&=\prod_{h\in G}\frac{f_{g^{-1}h}\left(x_{\sigma(h)}\right)}{f_h\left(x_{\sigma(h)}\right)}\\
 		&=\prod_{h\in G,\ \sigma(h)\neq h}\frac{f_{g^{-1}h}\left(x_{\sigma(h)}\right)}{f_{g^{-1}h}\left(x_h\right)}\frac{f_h\left(x_h\right)}{f_h\left(x_{\sigma(h)}\right)}\prod_{h\in G}\frac{f_{g^{-1}h}\left(x_h\right)}{f_h\left(x_h\right)}\\
 		&=\frac{g'(x)}{\sigma'(x)}\prod_{h\in G,\ \sigma(h)\neq h}\frac{f_{g^{-1}h}\left(x_{\sigma(h)}\right)}{f_{g^{-1}h}\left(x_h\right)}\\
 		&=\frac{g'(x)}{\sigma'(x)}\prod_{h\in G,\ \sigma(h)\neq h}\frac{f_{g^{-1}h}\left(x_{\sigma(h)}\right)}{f_{g^{-1}\sigma(h)}\left(x_{\sigma(h)}\right)},
 	\end{align*}
 	where the last equality comes from rearranging the terms in the denominators in the finite product. 
 	Since for all $u\in [0,1]$ the set $$\left\{g\in G:\ f_g(u)\neq 1\right\}$$ is finite,  we see that 
 	\[
 	\lim_{g\to\infty}\prod_{h\in G,\ \sigma(h)\neq h}\frac{f_{g^{-1}h}\left(x_{\sigma(h)}\right)}{f_{g^{-1}\sigma(h)}\left(x_{\sigma(h)}\right)}=1.
 	\]
 	Hence for all $\sigma \in\fix_G$ and  for almost all $x\in\Omega$, we have
 	\begin{equation}\label{eq: log RN relation}
 		\lim_{g\to\infty}\big[\log_\lambda(g'\circ\sigma(x))-\log(\sigma'(x))-\log_\lambda(g'(x))\big]=0.
 	\end{equation}
 	For $\tilde{\mu}$-almost every $(x,n)\in\Omega\times\bb{Z}$, we have
 	\begin{align*}
 		d_{\Omega \times \Z} \big[g(x,n),g(\sigma(x,n) )\big]&=d_{\Omega}(gx,g\sigma x) \ +  \\ &\quad \big|\log_\lambda(g'(x))-[\log_\lambda(g'(\sigma(x))-\log_\lambda(\sigma'(x))]\big|.	
 	\end{align*}
 	By Remark \ref{key-example}, the first term on the right tends to $0$ as $g\to\infty$.  The second one tends to $0$ by \eqref{eq: log RN relation}. 
 \end{proof}

\section{Concluding remarks}\label{conclude}

\subsection{The proof of Theorem \ref{countable}}

The proof of Theorem \ref{countable} follows from a routine modification of the proof of Theorem \ref{main-exists}, which we outline below.

\begin{proof}[Proof of Theorem \ref{countable}]
Let $\lambda \in (0,1)$ and consider the Bernoulli shift $(\Omega, \borel, \mu, T)$ given in Section \ref{def-c}.  By Remark \ref{int-den-p}, the functions $(f_n)_{n \in \Z^{+}}$ embedded in the definition of the product measure $\mu$ are densities with respect the underlying probability measure $\rho$.  

  The proof of Theorem \ref{countable} follows from replacing the underlying probability measure,  Lebesgue measure on $[0,1]$,  in Section \ref{proof-exists} with the {\em new} underlying measure $\rho$.  For example,  integrals with respect to Lebesgue measure in Lemma \ref{nonsing} become integrals (or weighted sums) with respect to
  the underlying measure  $\rho$  
  and in Lemma \ref{exch} we verify tameness with respect to $\rho$. 

With this substitution, the proof is the same.
	\end{proof}

\subsection{Type-$\mathrm{III}_1$ examples on $[0,1]^\mathbb{Z}$}
By considering a certain mixture
 of our previous densities from Section \ref{def-l} it is not difficult to write down  type-$\mathrm{III}_1$ Bernoulli shifts of a similar form.

Let $0<c<1<M$ and 
 consider a  sequence of functions $f_n:[0,1]\to (c,M)$ such that 
\[
\sum_{n=1}^\infty \int_0^1 \left(\sqrt{f_n}-\sqrt{f_{n-1}}\right)^2du<\infty
\] 
and 
\[
\sum_{n=1}^\infty \int_0^1 \left(\sqrt{f_n}-1\right)^2du=\infty.
\]
These conditions  imply that if we set $f_n=1$ for all $n<0$, then shift is a nonsingular $K$-automorphism with respect to $\mu=\bigotimes_{n\in\bb{Z}}f_n$ and the $\fix$-action is ergodic since the tameness condition of Aldous and Pitman   holds as the functions are uniformly bounded from above and below and Theorem \ref{AP} applies.

In addition, if   there exists $a<1$ such that 
\[
\int_{[0,1]} \left(\frac{1}{\left(T^n\right)'}\right)^2d\m=O(n^a),
\]
then the shift is conservative and ergodic. Using the argument with asymptotic pairs, in order that the shift will be type-$\mathrm{III}_1$,  it is sufficient that the $\fix$-action is of type-$\mathrm{III}_1$. The following two constructions satisfy all these conditions, and we state them without proof.   The first construction will be in the spirit of the constructions given for the proof of Theorem \ref{main-exists}, whereas the latter construction resembles constructions given in for a Bernoulli shift on two symbols.
\begin{example}
Let $0<\muu<\lambda<1$ be two numbers such that $\log(\muu)$ and $\log(\lambda)$ are linearly independent over $\mathbb{Q}$. Let $A_n,B_n,D_n,E_n$ be a decreasing sequence of disjoint intervals such that
\[
\mathcal{L}\left(A_n\right)=\mathcal{L}\left(D_n\right)=\lambda^{-1}\mathcal{L}\left(B_n\right)=\mu^{-1}\mathcal{L}\left(E_n\right)=\frac{1}{(n+4)\log (n+4)}.
\]
Set for $n\leq 1$, $f_n \equiv 1$.    For $n\geq 2$, let $f_n:[0,1]\to\left\{\delta,\lambda,1,\lambda^{-1},\delta^{-1}\right\}$ be given by  
$$
f_n(u)=\begin{cases}
\lambda, &\ u\in A_n,\\
\lambda^{-1}, &\ u\in B_n,\\
\muu,&\ u\in D_n,\\
\muu^{-1}, &\ u\in E_n,\\
1,&\ \text{otherwise}.
\end{cases}
$$
The proof of conservativity of the shift is similar  to the prove of Proposition \ref{cons}.   
With an argument similar to the proof of Theorem \ref{perm-type}, it follows that $\log(\muu)$ and $\log(\lambda)$ are essential values for the $\fix$-action.  Since the essential values are a closed subgroup of $\mathbb{R}$,
 the assumed rational independence 
 gives that  then the $\fix$-action is of type-$\mathrm{III}_1$.    \erk
\end{example}

\begin{example}
Set $$\lambda_n:=1-\frac{1}{\sqrt{n\log n}}.$$   Let $f_n:[0,1]\to \left\{\lambda_n,2-\lambda_n\right\}$ be defined by 
$$
f_n(u)=\begin{cases}
\lambda_n, &\ 0\leq u\leq \frac{1}{2},\\
2-\lambda_n, &\ \frac{1}{2}\leq u\leq 1.
\end{cases}
$$
It is not difficult to verify that the resulting Bernoulli shift will be of type-$\mathrm{III}_1$. \erk
\end{example}

\subsection{Power weakly mixing}
\label{pwm}

Recall that $T$ is \dff{power weakly mixing} if  for all $n_1,n_2,...,n_k\in\bb{Z}$, the product, $T^{n_1}\times T^{n_2}\times \cdots \times T^{n_k}$ is ergodic.  Power weakly mixing type-$\mathrm{III}$ Bernoulli shifts were constructed by Kosloff  \cite[Theorem 7]{KosZeroType}.  The Bernoulli shifts we constructed are also weakly power weakly mixing.

\begin{corollary}
	\label{PW}
	Let $T$  be the    Bernoulli shift given in Theorem \ref{main-exists}.   Then $T$ is  power weakly mixing.
\end{corollary}

\begin{proof}
	Let $\lambda\in (0,1)$ and $\mu=\bigotimes_{n\in\bb{Z}}f_n$ be the corresponding product measure and  $n_1,..,n_k\in\bb{Z}\setminus\{0\}$. Let $\mu^{\otimes k}$ be the $k$-fold product measure of $\mu$ and $S=T^{n_1}\times\cdots\times T^{n_k}$. Since $S$ is a direct product of $K$-automorphisms it is also a $K$-automorphism. It follows from  \eqref{eq: VW} that there exists $c=c(\lambda)>0$ such that for all $r\in\mathbb{N}$, we have
	\begin{align}
		\int_{\Omega^k}\left(\frac{1}{\left(S^r\right)'}\right)^2d\mu^{\otimes k}&=\prod_{j=1}^k\int_{\Omega}\left(\frac{1}{\left(T^{rn_j}\right)'}\right)^2d\mu  \nonumber\\
		\label{USE-MAH}
		&\leq \prod_{j=1}^k\left(\log\left(r\left|n_j\right|+1\right)\right)^{c}\leq \left(M+\log(r+1)\right)^{kc},
	\end{align}
	where $M:=\max_{j\in\{1,\ldots,k\}}\log\left(\left|n_j\right|\right)$. Since 
	\[
	\sum_{r=1}^\infty \frac{\left(M+\log(r+1)\right)^{kc}}{r^2}<\infty,
	\]
	a similar argument as in the proof of	Lemma \ref{cons} shows that $S$ is conservative. Since $S$ is a conservative $K$-automorphism, by Proposition \ref{ST-K} it is hence ergodic.  Hence $T$ is power weakly mixing. 
\end{proof}

 We will also adapt the  methods used to prove Lemma \ref{cons} to give a  sufficient criteria for a direct product of Maharam extensions to be conservative. 
\begin{proposition}\label{prop: Maharam prop}
	Let $\ns{ \left(\Omega_i,\F_i,\m_i,T_i\right) }_{i=1} ^k$  be $k$ nonsingular maps on probability spaces. If there exists $1<p<2$ such that 
	\begin{equation}\label{V is D}
		\sum_{n=1}^\infty n^{-p} \prod_{i=1}^k \int_{\Omega_i} \left(\frac{1}{\left(T_i^n\right)'}\right)^2d\m_i<\infty,
	\end{equation}
	then the $k$-fold direct product of the Maharam extensions is conservative. 
\end{proposition}

The following application of H\"older's inequality will be used in our proof of Proposition \ref{prop: Maharam prop}.
\begin{lemma}\label{Claim: V is D}
	Let $\left(\Omega,\mathcal{F},\m,T\right)$ be a nonsingular dynamical system with $\m(\X)=1$.   Then for all $0 <  \alpha\leq \beta$, we have
	\[
	1\leq \int_\X \left(T'\right)^{-\alpha}d\m\leq \int_\X \left(T'\right)^{-\beta}d\m,
	\]
	where the first equality is obtained if and only if $T$ is measure-preserving. 
\end{lemma}
\begin{proof}
	Set $\gamma=\alpha/(1+\alpha)$.   By H\"older's inequality, with $p=1/\gamma$ and $q=1/(1-\gamma)$, we have
	\begin{align*}
		1=\int_\X \left(T'\right)^{-\gamma}\left(T'\right)^\gamma d\m &\leq \left(\int_\X T'd\m\right)^{1/p}\left(\int_\X \left(T'\right)^{-\gamma q}d\m\right)^{1/q}\\
		&=\left(\int_\X \left(T'\right)^{-\alpha}d\m\right)^{1/q},
	\end{align*}
	since $\alpha=\gamma/(1-\gamma)$ and $\int_\X T'd\m=1$. 
	
	Recall that equality in  H\"older's inequality is obtained  if and only if there is $\delta>0$ such that $(T')^\gamma=\delta (T')^{-\gamma}$ almost everywhere which is equivalent to the requirement that  $$\sqrt{T'}=\delta^{\frac{1}{2\gamma}}(T')^{-1/2}.$$ Multiplying the above equations by $\sqrt{T'}$ on both sides we see that there exists $\delta$ such that $T'=\delta^{\frac{1}{2\gamma}}$ almost everywhere. Since $\m \circ T(\X)=1=\delta^{\frac{1}{2\gamma}} \m(\X)$, we deduce that $T$ is measure-preserving.  
	
	For the upper bound,  set $p=\beta/\alpha$.   By H\"older's inequality, 
	\[
	1\leq \int_\X \left(T'\right)^{-\alpha}d\m\leq \left(\int_\X \left(T'\right)^{-\beta}d\m\right)^{\alpha/\beta}\leq \int_\X \left(T'\right)^{-\beta}d\m.
	\]
\end{proof}
 In order to apply Hopf's  criteria in our proof of Proposition \ref{prop: Maharam prop}, we will replace the regular Maharam measure by an equivalent probability measure. 
Let $(\X,\F,\m)$ be a probability space.  Let $\alpha >0$ and     $\nua$ be  the probability measure on $\X\times\bb{R}$  given by
\begin{equation}
	\label{rep-Mm}
\nua(A \times I ) = \m(A) \int_I\frac{1}{2\alpha}e^{-\alpha|u|}du
\end{equation}
for $ A \in \F$ and an interval $I \subset \R$.

\begin{lemma}\label{Cl: For T}
	Let $(\X,\F,\m,T)$ be a nonsingular dynamical system and $\tilde{T}:\X\times\bb{R}\to \X\times\bb{R}$ its Maharam extension.  Let $\alpha >0$ and $\nu^{\alpha}$ be defined as in \eqref{rep-Mm}.   Then for $\nu^{\alpha}$-almost every $(x,u)$, we have
	\[
	\frac{d\nua\circ \tilde{T}}{d\nua}(x,u)=T'(x)\exp[-\alpha\left(\left|t+\log T'(x)\right|-|u|\right)], 
	\]
	where $T'(x)=\frac{dm\circ T}{dm}(x)$.  In particular, $\tilde{T}$ is nonsingular with respect to $\nua$.  
\end{lemma}

\begin{proof}
	The  lemma follows from a straightforward variation of the proof of the following familiar formula
	\begin{equation}
		\label{prod-RN}
		\frac{d(\rho_1 \otimes \rho_2)}{d(\m_1 \otimes \m_2)}(x_1,x_2) = \frac{d\rho_1}{d\mu_1}(x_1) \frac{d\rho_2}{d\mu_2}(x_2)
	\end{equation}
	that applies when $\rho_i$ is dominated by $\mu_i$.  
		\end{proof}

The proof of Proposition \ref{prop: Maharam prop} will make use of Lemmas \ref{Claim: V is D} and \ref{Cl: For T} and some elementary inequalities. 
In the course of the proof for each transformation there will be one probability measure associated to it and as usual $R'$ will denote the Radon-Nydodym derivative of $R$ with respect to its measure. 
\begin{proof}[Proof of Proposition \ref{prop: Maharam prop}]
	Let $\alpha=\frac{2-p}{p}$ such that $p(1+\alpha)=2$.   For each $1\leq i \leq k$,   let   $\nua_i$ be the probability measure on $\Omega_i\times \bb{R}$ be given by \eqref{rep-Mm}.   Let $\tilde{T}_i$ denote the Maharam extension of $T_i$.   Since $\nua_i$ is equivalent to the usual Maharam measure on $\Omega_i \times \R$ it suffices to show conservativity of the product transformation 
	$\mathbf{S}=\tilde{T_1}\times\cdots\times \tilde{T_k}$ on the space 
	$$\mathbf{\Omega} \times \R^k = \left(\Omega_1\times\bb{R}\right)\times\cdots \times \left(\Omega_k\times\bb{R}\right)$$ with respect to the measure  $\rho  = \nua_1 \otimes \cdots \otimes \nua_k$.

	 By Lemma \ref{Cl: For T}, for $\nua_i$-almost every $(w_i,u)\in\Omega_i\times\bb{R}$ and  for all $n\in\bb{N}$, we have 
\begin{align*}
\frac{d\nua_i\circ \tilde{T_i}^n}{d\nua_i}(w,u)&=\left( T_i^n\right)'(w_i)e^{-\alpha\left(\left|u+\log\left( T_i^n\right)'(w_i)\right|-|u|\right)}\\
&\geq \left( T_i^n\right)'(w_i)e^{-\alpha\left|\log\left( T_i^n\right)'(w_i)\right|}\\
&= \min\big(\left( T_i^n\right)'(w_i)^{1-\alpha},\left( T_i^n\right)'(w_i)^{1+\alpha}\big)\\
&=:A_i(w_i,n),
\end{align*}
	since $e^{-|a|} = \min\ns{ e^{-a}, e^a}$.  Notice that the final expression no longer depends on variable $u \in \R$.  
	Moreover,  applying formula \eqref{prod-RN} gives  that for   $\rho$-almost every $(\mathbf{w}, \mathbf{u}) = [ (w_1, \ldots, w_k), (u_1, \ldots, u_k)]  \in \mathbf{\Omega} \times \R^k$ and    for all $n\in\bb{N}$, we have 
	\begin{align*}
	\frac{dm\circ \mathbf{S}^n}{dm}(\mathbf{w}, \mathbf{u})\geq \prod_{i=1}^k 	A_i(w_i,n) 
	&=:  \mathbf{A}(\mathbf{w},n).
	\end{align*}
Let  $\m =  \m_1 \otimes \cdots  \otimes \m_k$.  The Hopf criteria  \cite[Proposition 1.3.1]{AaroBook} together with the last inequality  tells us that if for $\mu$-almost every $\mathbf{w} \in \mathbf{\Omega}$, we have
	\begin{equation}\label{V is D sufficient cond}
		\sum_{n=1}^\infty \mathbf{A}(\mathbf{w},n)=\infty,
	\end{equation}
	then $\mathbf{S}$ is conservative with respect to $\rho$, as desired.

	It remains to verify \eqref{V is D sufficient cond}.   For all $n\in\bb{N}$, we have
	\begin{align*}
		\int_{\mathbf{\Omega}} &\left(\mathbf{A}(\mathbf{w},n)\right)^{-p}d\mu(\mathbf{w}) =\prod_{i=1}^k \int_{\Omega_i} A_i(w_i,n)^{-p}d\m_i(w_i) \nonumber\\
		&=\int_{\Omega_i} \max\left(\left(\frac{1}{\left( T_i^n\right)'(w_i)}\right)^{p(1-\alpha)},\left(\frac{1}{\left( T_i^n\right)'(w_i)}\right)^{p(1+\alpha)}\right) d\m_i(w_i) \\
		&\leq\prod_{i=1}^k \int_{\Omega_i}\left(\left(\frac{1}{\left( T_i^n\right)'(w_i)}\right)^{p(1-\alpha)}+\left(\frac{1}{\left( T_i^n\right)'(w_i)}\right)^{p(1+\alpha)}\right)d\m_i(w_i).\nonumber
	\end{align*}
	Since $0 <p(1-\alpha)<p(1+\alpha)=2$,  by Lemma \ref{Claim: V is D}, for all $1\leq i\leq k$, we have
	\begin{align*}
		\int_{\Omega_i}\left(\frac{1}{\left( T_i^n\right)'(w_i)}\right)^{p(1-\alpha)}d\m_i(w_i)& \leq  \int_{\Omega_i}\left(\frac{1}{\left( T_i^n\right)'(w_i)}\right)^{p(1+\alpha)}d\m_i(w_i)\\
		&=\int_{\Omega_i}\left(\frac{1}{\left( T_i^n\right)'(w_i)}\right)^{2}d\m_i(w_i).
	\end{align*}
	Hence
	\begin{align}
		\int_{\mathbf{\Omega}} \left( \mathbf{A}(\mathbf{w},n)\right)^{-p}d\mu(\mathbf{w})&\leq 2^k \prod_{i=1}^k\int_{\Omega_i}\left(\frac{1}{\left( T_i^n\right)'(w_i)}\right)^{2}d\m_i(w_i).
	\end{align}
	Setting
	\[
	B_n=\left\{\mathbf{w} \in\mathbf{\Omega}: \mathbf{A}(\mathbf{w},n)\leq n^{-1}\right\}=\left\{\mathbf{w}\in\mathbf{\Omega}: \mathbf{A}(\mathbf{w},n)^{-p}\geq n^p\right\}
	\]
	it follows from Markov's inequality that
	\begin{align*}
		\sum_{n=1}^\infty \mu\left(B_n\right)&\leq \sum_{n=1}^\infty n^{-p}\int_{\mathbf{\Omega}} \left(\mathbf{A}(\mathbf{w},n)\right)^{-p}d\mu\\
		&\leq 2^k\sum_{n=1}^\infty n^{-p}\prod_{i=1}^k\int_{\Omega_i}\left(\frac{1}{\left( T_i^n\right)'(w_i)}\right)^{2}d\m_i(w_i)<\infty.
	\end{align*}
	Thus as in the proof of Lemma \ref{cons},
	 the Borel-Cantelli lemma and a simple comparison test implies  \eqref{V is D sufficient cond}.
\end{proof}

We will apply Proposition \ref{prop: Maharam prop} to the Bernoulli shifts in Theorem \ref{main-exists}, but first we  will need a version for the discrete Maharam extension.  While one can reiterate the same proof with minor modifications we argue via the following well-known lemma. 

\begin{lemma}\label{lem: classical}
	Let $\ns{ \left(\Omega_i,\F_i,\m_i,T_i\right) }_{i=1} ^k$  be $k$ nonsingular maps on probability spaces. Suppose that for some $0<\lambda<1$, for all $1\leq i \leq k$, we have  
	\[
	\log_{\lambda}(T_i)'\in \Z \ \  \text{almost surely.}
	\]
	Then the $k$-fold direct product of the Maharam extensions is conservative if and only if the $k$-fold direct product of the discrete  Maharam extensions of the $\Z$-actions is conservative. 
\end{lemma}
\begin{proof} 
Let $T$, $S$ and $R$ be the $k$-fold direct product of the $T_i$, the $k$-fold direct product of the Maharam extensions and the $k$-fold direct product of the discrete  Maharam extensions, respectively. Then $S$ is isomorphic to the $\bb{R}^k$-skew product extension of $T$ by the cocycle
\[
\alpha\left(n,\left(w_1,w_2,\ldots,w_n\right)\right)=\left(\log\left(T_1^n\right)'\left(w_1\right),\ldots,\log\left(T_k^n\right)'\left(w_k\right)\right).
\]
Similarly, $R$ is isomorphic to the $\bb{Z}^k$-skew product extension of $T$ by the cocycle
\[
\beta\left(n,\left(w_1,w_2,\ldots,w_n\right)\right)=\left(\log_\lambda\left(T_1^n\right)'\left(w_1\right),\ldots,\log_\lambda\left(T_k^n\right)'\left(w_k\right)\right).
\]
By \cite[Theorem 5.5]{KS-type}, $S$ is conservative if and only if $\mathbf{0}\in\bb{R}^k$ is an essential value for $(T,\alpha)$; that is, writing $\mu=\otimes_{i=1}^k\mu_i$, for every $A\in \F_1\otimes\cdots \otimes\F_k$ with $\mu(A)>0$ and $\epsilon>0$ there exists $n\in\bb{Z}$ such that 
\[
\mu\left(A\cap T^{-n}A\cap \left[|\alpha(n,\cdot)|<\epsilon\right]\right)>0.
\]
Similarly, $R$ is conservative if and only if $\mathbf{0}\in\bb{Z}^k$ is an essential value of $(T,\beta)$. Since $\beta=(\log \lambda)^{-1}\alpha$, the vector  $\mathbf{0}$ is an essential value for $(T,\alpha)$ if and only if $\mathbf{0}$ is an essential value for $(T,\beta)$ and we conclude that $S$ is conservative if and only if $R$ is conservative.
\end{proof}

\begin{corollary}
	\label{cor-fixed}
	Let $T$  be the Bernoulli shift for Theorem \ref{main-exists} defined in Section \ref{def-l}.  Then the discrete  Maharam extension of $T$ is a power weakly mixing $K$-automorphism.
\end{corollary}
\begin{proof}
	Let $n_1,n_2,...,n_k\in\bb{Z}$.  By Theorem \ref{duality-checked}, the product $S=\tilde{T}^{n_1}\times \tilde{T}^{n_2}\times \cdots \times \tilde{T}^{n_k}$  is a $k$-fold product of  $K$-automorphisms, hence is a $K$-automorphism.  By \eqref{USE-MAH}, there exists $C,M>0$ such that
	\[
	\sum_{n=1}^\infty n^{-1.1}\prod_{j=1}^k\int_{\Omega}\left(\frac{1}{\left(T^{rn_j}\right)'}\right)^2d\mu\leq 	\sum_{n=1}^\infty n^{-1.1}(M+\log(n+4))^C<\infty.
	\] 
	 Proposition \ref{prop: Maharam prop} with $p=1.1$, together with Lemma \ref{lem: classical} give that $S$ is conservative. By Proposition \ref{ST-K}, a conservative $K$-automorphism is ergodic, hence $\tilde{T}^{n_1}\times \tilde{T}^{n_2}\times \cdots \times \tilde{T}^{n_k}$ is ergodic. 
\end{proof}

\begin{remark}
	In a previous version of this manuscript, Stefaan Vaes informed us that our previous approach to proving Corollary \ref{cor-fixed} was  problematic.  The corrected approach, given here, relies more specifically on the properties of the Bernoulli shifts given for the proof of Theorem \ref{main-exists}.  \erk
	\end{remark}

\subsection{Type-$\mathrm{III}_0$ Bernoulli shifts}

Type-$\mathrm{III}_0$ systems are much less understood than the other type-$\mathrm{III}$ systems and have not been treated in  this paper.    It is known that one can construct product odometers \cite{MR796751, MR1646621} that are of type-$\mathrm{III}_0$, but little is known about the possibility of a Bernoulli shift of this type.  

\begin{question} 
	\label{QZ}
	 Does there exists a Bernoulli shift that is of type-$\mathrm{III}_0$?
 	\end{question}

\begin{remark}
	Tey Berendschot and Stefaan Vaes informed us that they have also been considering similar problems and they have answered Question \ref{QZ} positively; their results are now available and  presented in   \cite{berendschot2020nonsingular}. \erk
	\end{remark}

\section*{Acknowledgments} 
Zemer Kosloff is funded in part by ISF grant No.\ 1570/17.  We thank the referee for  helpful comments and attention to detail.

 \bibliographystyle{abbrv}
   	\bibliography{embedding}

\end{document}